\documentclass[a4paper]{amsart}

\usepackage[utf8x]{inputenc}

\usepackage{amsmath,amssymb,enumerate,amsthm,url,graphicx}
\usepackage{amsrefs}
\usepackage[pagewise,displaymath]{lineno}

\usepackage{wasysym}                            

\newcommand{\UP}{\blacktriangle}                
\newcommand{\DOWN}{\blacktriangledown}          
\newcommand{\Up}{{\vartriangle}}                
\newcommand{\Down}{\triangledown}               

\newcommand{\LEFT}{\LHD}
\newcommand{\RIGHT}{\RHD}
\newcommand{\Left}{\lhd}
\newcommand{\Right}{\rhd}

\newtheorem{theorem}{Theorem}[section]
\newtheorem{lemma}[theorem]{Lemma}
\newtheorem{corollary}[theorem]{Corollary}
\newtheorem{proposition}[theorem]{Proposition}

\theoremstyle{definition}
\newtheorem{remark}[theorem]{Remark}
\newtheorem{example}[theorem]{Example}

\newcommand{\TABS}{%
\quad\=\qquad\qquad\qquad\qquad\qquad\qquad\qquad\qquad\= \kill}

\title[Intuitionistic Logic with two Galois Connections]%
{Intuitionistic logic with two Galois connections combined with Fischer Servi axioms}

\author{Wojciech Dzik}
\address{{Wojciech Dzik},~Institute of  Mathematics, University of Silesia, ul.~Bankowa~12, \mbox{40-007 Katowice}, Poland}
\email{dzikw@silesia.top.pl}

\author{Jouni J{\"a}rvinen}
\address{{Jouni J{\"a}rvinen},~Turku, Finland}
\email{Jouni.Kalervo.Jarvinen@gmail.com}

\author{Michiro Kondo}
\address{{Michiro Kondo},~School of Information Environment, Tokyo Denki University, Inzai, 270-1382, Japan}
\email{kondo@sie.dendai.ac.jp}

\date{\today}

\begin{document}

\maketitle

\begin{abstract}
Earlier, the authors introduced the logic IntGC, which is an extension of intuitionistic 
propositional logic by  two rules of inference mimicking the performance of Galois connections 
(Logic J. of the IGPL, 18:837-858, 2010).
In this paper, the extensions Int2GC and Int2GC+FS of IntGC are studied. Int2GC can be seen
as a fusion of two IntGC logics, and Int2GC+FS is obtained from Int2GC by adding instances of duality-like connections 
$\Diamond( A \to\ B) \to (\Box A \to \Diamond B)$ and $(\Diamond A \to \Box B) \to \Box(A \to B)$,
introduced by G.~Fischer Servi (Rend. Sem. Mat. Univers. Politecn. Torino, 42:179-194, 1984),
for interlinking the two Galois connections of Int2GC. 
Both Kripke-style and algebraic semantics are presented for Int2GC and Int2GC+FS, 
and the logics are proved to be complete with respect to both of these semantics.
We show that rough lattice-valued fuzzy sets defined on complete Heyting algebras are proper algebraic 
models for Int2GC+FS. We also prove that Int2GC+FS is equivalent to the intuitionistic tense logic IK$_t$, and
an axiomatisation of IK$_t$ with the number of axioms reduced to the half 
of the number of axioms given by W.~B.~Ewald (J. Symb. Log, 51:166–179, 1986) is presented.
\end{abstract}


\section{Introduction and Motivation} \label{Section:Intro}

In \cite{JaKoKo08}, Information Logic of Galois Connections ({\sf ILGC}) 
was introduced as classical propositional logic with a pair 
of unary connectives $\UP$ and $\Down$ mimicking a Galois connection. 
Motivation for {\sf ILGC} originates in \emph{rough set theory}
\cite{Pawl82}, where it is assumed that our knowledge about 
objects of a universe of discourse $U$ is expressed by an 
\emph{information relation} $R$. 
An information relation may reflect similarity or difference
between objects. For instance, $R$ can be defined on the
set of all human beings in such a way that two persons are 
$R$-related if they are of the same gender and the difference of
their ages is less than a year. Originally, Pawlak assumed information
relations to be equivalences (reflexive, symmetric, and
transitive binary relations), so called \emph{indiscernibility
relations}, but in the literature can be found numerous
studies considering information relations of different
type; see \cite{DemOrl02}, for example.

In terms of an information relation $R$, we may define
the \emph{upper approximation\/} of a set $X \subseteq U$ as
\[
X^\UP = \{ x \in U \mid (\exists y \in U)\, x R y \ \& \ y \in X \},
\]
and the \textit{lower approximation\/} of $X$ is
\[
X^\DOWN = \{x \in U \mid (\forall y \in U)\, x R y \Rightarrow y \in X \}.
\]
For instance, if $R$ is the information relation considered above,
then $x \in X^\DOWN$ if all the persons that are coarsely of the same age
and are of the same gender as $x$ belong to $X$, and $x \in X^\UP$ if there
exists at least one such person. Therefore, $^\DOWN$ 
may be interpreted to represent certainty and $^\UP$
possibility with respect to knowledge expressed by 
the relation $R$.

We may also define another pair of mappings $\wp(U) \to \wp(U)$ 
by reversing the relation $R$. For any set $X \subseteq U$, 
let us define
\[
X^\Up = \{x \in U \mid (\exists y \in U)\, y R x \ \& \ y \in X \}
\]
and 
\[
X^\Down = \{x \in U \mid (\forall y \in U)\, y R x \Rightarrow y \in X \}.
\]
It is well-known that for any binary relation, the pairs 
$(^\UP,^\Down\!)$ and $(^\Up,^\DOWN\!)$ are order-preserving Galois connections
$\wp(U) \to \wp(U)$.

The logic {\sf ILGC} was defined by adding to classical propositional logic 
two rules of inference:\medskip
\begin{tabbing}
\TABS 
\>(GC\,${\Down}{\UP}$) \ $\displaystyle \frac{A \to \Down B}{\UP A \to B}$ 
\>(GC\,${\UP}{\Down}$) \ $\displaystyle \frac{\UP A \to B}{A \to \Down B}$ 
\end{tabbing} \medskip
Another pair of connectives is introduced by De~Morgan-type assertions: 
\begin{equation}\tag{$\star$}\label{Eq:dual}
\Up A = \neg \Down \neg A \mbox{\quad and \quad} \DOWN A = \neg \UP \neg A,
\end{equation}
For $\Up $ and $\DOWN $ the following rules are admissible in ILGC:\medskip
\begin{tabbing}
\TABS
\>(GC\,${\DOWN}{\Up}$) \ $\displaystyle \frac{A \to \DOWN B}{\Up A \to B}$ 
\>(GC\,${\Up}{ \DOWN}$) \ $\displaystyle \frac{\Up A \to B}{A \to \DOWN B}$
\end{tabbing}\medskip
This means that in {\sf ILGC}, we get another Galois connection $({\Up},{\DOWN})$ ``for free''.

In \cite{DzJaKo10}, we introduced an intuitionistic propositional logic
with a Galois connection ({\sf IntGC}) and studied its main properties. In addition to the intuitionistic 
logic axioms and inference rule of Modus Ponens, {\sf IntGC} contains 
rules (GC\,${\Down}{\UP}$) and (GC\,${\UP}{\Down}$). Since the base logic is changed from 
classical to intuitionistic, the classical-type assertions \eqref{Eq:dual} 
can not be used to introduce another Galois connection. More precisely, if we define the operators $\Up$, $\DOWN$ 
from $\Down$, $\UP$ in terms of intuitionistic negation, the pair $({\Up},{\DOWN})$ 
does not form a Galois connection; see Lemma 3.3 in \cite{DzJaKo10}.
Therefore, to define an intuitionistic logic of two Galois connections, 
the other Galois connection must be declared by adding rules (GC\,${\DOWN}{\Up}$) and 
(GC\,${\Up}{\DOWN}$).

In Section~\ref{Sec:Axiomatizations}, we define two intuitionistic logics with two Galois connections.
The first one, called {\sf Int2GC}, is obtained by extending intuitionistic propositional 
logic with the connectives $\UP$, $\DOWN$, $\Up$, $\Down$ and by rules 
(GC\,${\Down}{\UP}$), (GC\,${\UP}{\Down}$), (GC\,${\DOWN}{\Up}$), (GC\,${\Up}{\DOWN}$).
In {\sf Int2GC}, the two Galois connections $({\UP},{\Down})$ and $({\Up},{\DOWN})$
are not connected with each other, and this means that {\sf Int2GC} is simply 
the fusion of two {\sf IntGC} logics, the first one having the operators $\UP$ and $\Down$, 
and the second has $\Up$ and $\DOWN$. 
The logic {\sf Int2GC+FS} is obtained by extending {\sf Int2GC} with instances
of the axioms $\Diamond( A \to\ B) \to (\Box A \to \Diamond B)$ and
$(\Diamond A \to \Box B) \to \Box(A \to B)$ introduced by
Fischer Servi \cite{FishServ84}. This means that {\sf Int2GC+FS} has
duality-like connections  $\UP(A \to B) \to (\DOWN A \to \UP B)$, 
$\Up(A \to B) \to (\Down A \to \Up B)$, $(\UP A \to \DOWN B) \to \DOWN(A \to B)$, 
and $(\Up A \to \Down B) \to \Down(A \to B)$. These axioms defining {\sf Int2GC+FS}
are referred to as (FS1), (FS2), (FS3), and (FS4), respectively. We show that  
in {\sf Int2GC}, axioms (FS1) and {\rm (FS4)} are equivalent, and the
same holds with (FS2) and (FS3). This implies that we have several
equivalent combinations of axioms to define {\sf Int2GC+FS}.

Section~\ref{Sect:H2GC-algebras} is devoted to H2GC- and H2GC+FS-algebras that are used for
defining algebraic semantics for {\sf Int2GC} and {\sf Int2GC+FS}, respectively. 
H2GC-algebras are Heyting algebras equipped with two order-preserving Galois connections 
$({^\RIGHT}, {^\Left})$ and $({^\Right},{^\LEFT})$, and H2GC+FS-algebras are H2GC-algebras
such that the operations ${^\RIGHT}$ and ${^\LEFT}$ are connected by an identity corresponding 
to axiom (FS1), and ${^\Right}$ and ${^\LEFT}$ are connected by an equation that corresponds (FS2).
In \cite{Dunn94}, J.~M.~Dunn studied distributive lattices with two operators $\Box$ and $\Diamond$. 
He introduced conditions ($\mathrm{D}_\wedge$)~~$\Diamond x \wedge \Box y \leq \Diamond( x \wedge y)$ \ and \  
($\mathrm{D}_\vee$)~~$\Box(x \vee y) \leq \Box x \vee \Diamond y$
for the interactions between $\Box$ and $\Diamond$. We show that H2GC+FS-algebras 
can be defined also as H2GC-algebras satisfying the identities corresponding ($\mathrm{D}_\wedge$),
that is, H2GC+FS-algebras are H2GC-algebras satisfying  $a^\RIGHT \wedge b^\LEFT \to ( a \wedge b)^\RIGHT = 1$ 
and $a^\Right \wedge b^\Left \to ( a \wedge b)^\Right = 1$. In Section~\ref{sSec:FuzzyMotivation},
we consider rough fuzzy sets defined on complete Heyting algebras, and show how in this setting
H2GC+FS-algebras arise naturally. Algebras of rough fuzzy sets
satisfy ($\mathrm{D}_\wedge$) when $\Diamond$ and $\Box$ are interpreted by
$^\RIGHT$ and $^\LEFT$ (or $^\Right$ and $^\Left$), but condition ($\mathrm{D}_\vee$) is not satisfied.
So, rough fuzzy sets are proper algebraic models for {\sf Int2GC+FS}. In Section~\ref{sSec:AlgCompleteness},
we introduce algebraic semantics for {\sf Int2GC} and {\sf Int2GC+FS}
with respect to H2GC- and H2GC+FS-algebras, respectively, and 
present algebraic completeness theorems.

In Section~\ref{Sec:Kripke}, Kripke-semantics for {\sf Int2GC} and {\sf Int2GC+FS}
are considered. We begin with recalling Kripke-frames and completeness
for {\sf IntGC} from \cite{DzJaKo10} in Section~\ref{sSec:Frames}. 
In addition, we introduce Kripke-frames and semantics for {\sf Int2GC}
and {\sf Int2GC+FS}, and soundness of both  {\sf Int2GC} and {\sf Int2GC+FS} is proved. 
Canonical frames of H2GC-algebras are introduced and Kripke-completeness is proved. 
Section~\ref{sSec:Frames} ends by an example in which particular Kripke-frames for 
{\sf Int2GC+FS} are defined in terms of preference relations. In Section~\ref{sSec:Canonical}, 
we define canonical frames of H2GC+FS-algebras, and Kripke-completeness of {\sf Int2GC+FS}
is proved by applying canonical frames and algebraic completeness result of {\sf Int2GC+FS}.

It is proved in \cite{JaKoKo08} that {\sf ILGC} is equivalent, with respect to provability, 
to the minimal (classical) tense logic {\sf K$_t$}, that is, {\sf ILGC} can be viewed as a simple formulation of 
{\sf K$_t$}. In Section~\ref{Sec:IKt}, we prove that intuitionistic tense logic {\sf IK$_t$},
introduced by Ewald \cite{Ewald86}, is equivalent syntactically to {\sf Int2GC+FS} when 
$\UP$, $\DOWN$, $\Up$, $\Down$ are identified with tense operators $F$, $G$, $P$, $H$, respectively.
In other words, in {\sf IK$_t$} and {\sf Int2GC+FS} exactly the same formulas can be proved. 
This then means that {\sf Int2GC+FS} can be seen as an alternative formulation of {\sf IK$_t$}.
In addition, we give an axiomatisation of {\sf IK$_t$} with the number of axioms reduced to half 
of the number of axioms of {\sf IK$_t$} (with the same rules) given by  Ewald \cite{Ewald86}, and we present
another definition of {\sf Int2GC} using only axioms of Ewald and rules admissible in {\sf IK$t$}.

\section{Intuitionistic logics with Galois connections and Fischer Servi axioms}
\label{Sec:Axiomatizations}

In this section we introduce two  modal logics {\sf Int2GC} and {\sf Int2GC+FS} based on 
\emph{intuitionistic propositional logic} \cites{Dalen01,RasSik68}.
We begin with recalling the intuitionistic propositional logic with a Galois connection ({\sf IntGC})
defined by the authors in \cite{DzJaKo10}. The language of {\sf IntGC} is constructed from an enumerable
infinite set of propositional variables $\mathrm{Var}$, the connectives $\neg$, $\vee$, $\wedge$, $\to$,
and the unary operators $\UP$ and $\Down$. The constant \emph{true} is defined by setting
$\top := p \to p$ for some fixed propositional variable $p \in \mathrm{Var}$, and the constant 
\emph{false} is defined by $\bot := \neg \top$. We also set $A \leftrightarrow B := (A \to B) \wedge (B \to A)$.
The logic {\sf IntGC} is the smallest logic that contains intuitionistic 
propositional logic, and is closed under the rules of \emph{substitution},
\emph{modus ponens}, and rules (GC\,${\Down}{\UP}$) and (GC\,${\UP}{\Down}$). 
The following rules are admissible in {\sf IntGC}:

\begin{tabbing}
\TABS
\> (RN$\Down$) $\displaystyle \frac{A}{\Down A}$ \\[4mm]
\> (RM$\Down$) $\displaystyle \frac{A \to B}{\Down A \to \Down B}$  \>(RM$\UP$) $\displaystyle\frac{A \to B}{\UP A \to \UP B}$ \\
\end{tabbing}
In addition, the following formulas are provable:
\begin{enumerate}[\rm ({GC}1)]
\item $A \to \Down \UP A$ \ and \  $\UP \Down A \to A$;

\item $\UP A \leftrightarrow \UP \Down \UP A$ \ and \ $\Down A \leftrightarrow \Down \UP \Down A$;

\item $\Down \top$  \ and \  $\neg \UP \bot$;

\item $\Down (A \wedge B)  \leftrightarrow  \Down A \wedge \Down B$ \ and \
$\UP (A \vee B)  \leftrightarrow \UP A \vee \UP B$;

\item $\Down (A \to B) \to (\Down A \to \Down B)$.
\end{enumerate}

The language of the logic {\sf Int2GC} is the one of {\sf IntGC} extended by two unary connectives 
${\Up}$ and ${\DOWN}$, and the logic {\sf Int2GC} is the smallest logic extending {\sf IntGC} by rules
(GC\,${\DOWN}{\Up}$) and (GC\,${\Up}{\DOWN}$). Obviously, in {\sf Int2GC} also the rules:
\begin{tabbing}
\TABS
\> (RN$\DOWN$) $\displaystyle \frac{A}{\DOWN A}$ \\[4mm]
\> (RM$\DOWN$) $\displaystyle \frac{A \to B}{\DOWN A \to \Down B}$  \>(RM$\Up$) $\displaystyle\frac{A \to B}{\Up A \to \Up B}$ \\
\end{tabbing}
are admissible, and the following formulas are provable:
\begin{enumerate}[({GC}1)$^\star$]

\item $A \to \DOWN \Up A$ \ and \  $\Up \DOWN A \to A$;

\item $\Up A \leftrightarrow \Up \DOWN \Up A$ \ and \ $\DOWN A \leftrightarrow \DOWN \Up \DOWN A$;

\item $\DOWN \top$  \ and \  $\neg \Up \bot$;

\item $\DOWN (A \wedge B)  \leftrightarrow  \DOWN A \wedge \DOWN B$ \ and \
$\Up (A \vee B)  \leftrightarrow \Up A \vee \Up B$;

\item $\DOWN (A \to B) \to (\DOWN A \to \DOWN B)$.
\end{enumerate}

\emph{Intuitionistic modal logic}\/ {\sf IK} was introduced by
G.~Fischer Servi in \cite{FishServ84}. The logic {\sf IK} is obtained
by adding two modal connectives $\Diamond$ and $\Box$ to intuitionistic
logic satisfying the following axioms:
\begin{enumerate}[\rm ({IK}1)]\label{Axioms:Fischer}
\item $\Diamond (A \vee B) \to \Diamond A \vee \Diamond B$
\item $\Box A \wedge \Box B \to \Box (A \wedge B)$
\item $\neg \Diamond \bot$
\item $\Diamond( A \to\ B) \to (\Box A \to \Diamond B)$
\item $(\Diamond A \to \Box B) \to \Box(A \to B)$ 
\end{enumerate}
In addition, the \emph{monotonicity rules} for both $\Diamond$ and $\Box$ are admissible, that is:
\begin{tabbing}
\TABS
\> (RM$\Diamond$) $\displaystyle \frac{A \to B}{\Diamond A \to \Diamond B}$  
\> (RM$\Box$) $\displaystyle\frac{A \to B}{\Box A \to \Box B}$ 
\end{tabbing}

In this work, we call axioms (IK4) and (IK5) the \emph{Fischer Servi axioms}, and they have
a special role in interlinking the two Galois connections of {\sf Int2GC}. From (IK4) and (IK5)
we can form the following four axioms by replacing $\Box$ and $\Diamond$ by $\DOWN$ and $\UP$,
and by $\Down$ and $\Up$, respectively:
\begin{enumerate}[({FS}1)]
 \item  $\UP(A \to B) \to (\DOWN A \to \UP B)$ 
 \item $\Up(A \to B) \to (\Down A \to \Up B)$ 
 \item $(\UP A \to \DOWN B) \to \DOWN(A \to B)$ 
 \item $(\Up A \to \Down B) \to \Down(A \to B)$ 
\end{enumerate}

\begin{proposition} \label{Prop:FSEquiv}
In {\sf Int2GC}, the following assertions hold:
\begin{enumerate}[\rm (a)]
\item Axioms {\rm (FS1)} and {\rm (FS4)} are equivalent.
\item Axioms {\rm (FS2)} and {\rm (FS3)} are equivalent.
\end{enumerate}
\end{proposition}

\begin{proof}
We prove only assertion (a), because (b) can be proved analogously. Here
$\vdash A$ denotes that $A$ is provable in {\sf Int2GC}.

\smallskip
\noindent%
(FS1)$\Rightarrow$(FS4):  Let us set $X := A$, $Y := \DOWN \Up A$ and $Z := \UP \Down B$ 
in the provable formula  $(X \to Y) \to ((Y \to Z) \to (X \to Z))$.
We get $\vdash (\DOWN \Up A \to \UP \Down B) \to (A \to \UP \Down B)$
by using also $\vdash A \to \DOWN \Up A$. This is equivalent 
to $\vdash A \wedge (\DOWN \Up A \to \UP \Down B) \to \UP \Down B$. 
Because $\vdash \UP \Down B \to B$, this means  
$\vdash A \wedge (\DOWN \Up A \to \UP \Down B)  \to B$ and
$\vdash (\DOWN \Up A \to \UP \Down B)  \to(A \to B)$. 
If we set $A:= \Up A$ and $B:= \Down B$ in (FS1), we obtain
$\vdash \UP(\Up A \to \Down B) \to (\DOWN \Up A \to \UP \Down B)$,
and so $\vdash \UP(\Up A \to \Down B)  \to (A \to B)$. This implies
$\vdash (\Up A \to \Down B)  \to \Down(A \to B)$ by (GC\,${\UP}{\Down}$).

\smallskip\noindent%
(FS4)$\Rightarrow$(FS1): 
We set $X := \Up \DOWN A$, $Y := A$ and $Z := B$ 
in $(X \to Y) \to ((Y \to Z) \to (X \to Z))$. This gives
$\vdash  (\Up \DOWN A \to A) \to ((A \to B) \to (\Up \DOWN A \to B))$,
and $\vdash (A \to B) \to (\Up \DOWN A \to \Down \UP B)$,
since $\vdash \Up \DOWN A \to A$ and $\vdash B \to \Down \UP B$.
By monotonicity, $\vdash \UP(A \to B) \to \UP(\Up \DOWN A \to \Down \UP B)$. 
By setting $A := \DOWN A$ and  $B:= \UP B$ in (FS4), we have
$\vdash (\Up \DOWN A \to \Down \UP B) \to \Down (\DOWN A \to \UP B)$ and
$\vdash \UP(\Up \DOWN A \to \Down \UP B) \to (\DOWN A \to \UP B)$ 
by (GC\,${\Down}{\UP}$). Therefore, we obtain 
$\vdash \UP(A \to B) \to  (\DOWN A \to \UP B)$. 
\end{proof}

The logic {\sf Int2GC+FS} is defined as the extension of {\sf Int2GC} that
satisfies also the Fischer Servi axioms (FS1)--(FS4). By 
Proposition~\ref{Prop:FSEquiv} it is clear that we have several equivalent 
axiomatisations of {\sf Int2GC+FS}, that is:
\[ \mbox{\sf Int2GC+FS} = {\sf Int2GC} + \{ {\rm (FS1) \ \text{or} \ (FS4)} \}
+ \{ {\rm (FS2)  \ \text{or} \ (FS3)} \}. \]
The logic {\sf Int2GC+FS} satisfies the counterparts of axioms 
(IK1)--(IK5) of {\sf IK}, so {\sf Int2GC+FS} can be regarded as a intuitionistic bi-modal 
logic, and the pairs $\UP$, $\DOWN$ and  $\Up$, $\Down$ are intuitionistic modal connectives
in the sense of Fischer Servi.

\section{Algebraic Semantics and Completeness }  
\label{Sect:H2GC-algebras}

\subsection{H2GC-algebras and H2GC+FS-algebras}
\label{SubSec:AlgDefinitions}

In \cite{DzJaKo10}, we introduced HGC-algebras as counterparts of the logic {\sf IntGC},
and we showed that {\sf IntGC} is complete with respect to HGC-algebras. In this section, we
define H2GC- and H2GC+FS-algebras and give completeness theorems for {\sf Int2GC} and {\sf Int2GC+FS} 
in terms of these algebras.

Let $\varphi \colon P \to Q$ and $\psi \colon Q \to P$ be maps between
ordered sets $P$ and $Q$. The pair $(\varphi,\psi)$ is a
\emph{Galois connection} between $P$ and $Q$, if for all $p \in P$ and $q \in Q$,
\[
 \varphi(p) \leq q  \iff p \le \psi(q).
\]
For a Galois connection $(\varphi,\psi)$,
$\varphi$ preserves all existing joins and $\psi$ preserves all
existing meets. If $P$ and $Q$ are bounded, then $\varphi(0) = 0$ and $\psi(1) = 1$. In addition, 
a pair $(\varphi,\psi)$ forms a Galois connection if
and only if the following conditions hold:
\begin{enumerate}[\rm (i)]
\item $p \leq \psi (\varphi (p))$ for all $p \in P$ and $\varphi ( \psi(q)) \leq q$ for all $q \in Q$;
\item the maps $\varphi$ and $\psi$ are order-preserving.
\end{enumerate}
In particular, if $\varphi$ and $\psi$ are maps on a lattice $L$, then
the pair $(\varphi,\psi)$ is a Galois connection on $L$ if and only if the following identities are satisfied
for all $a,b \in L$:
\begin{enumerate}[({gc}1)]
 \item $\varphi(a \vee b) = \varphi(a) \vee \varphi(b)$ \ and \ $\psi(a \wedge b) = \psi(a) \wedge \psi(b)$
 \item $a = a \wedge \psi (\varphi (a))$ \  and  \ $a = a \vee \varphi(\psi(a))$  
\end{enumerate}
More properties of Galois connections can be found in \cite{ErKoMeSt93}, for instance.

A \emph{Heyting algebra} $H$ is a lattice with $0$ such that for all $a,b \in H$, 
there is a greatest element $x$ of $H$ with $a \wedge x \leq b$.
This element is the \emph{relative pseudocomplement} of $a$ with respect to $b$,
and is denoted $a \to b$. Note that Heyting algebras are always distributive pseudocomplemented
lattices such that the pseudocomplement $\neg a$ of $a$ is $a \to 0$. 
Because $\neg 0$ is the greatest element, Heyting algebras
are bounded. Thus, a Heyting algebra $H$ can be considered
as an algebraic structure $\mathbb{H} = (H,\vee,\wedge,\to, 0)$, which
can be equationally defined as follows (see e.g. \cite{BaDw74}):
\begin{enumerate}[({h}1)]
 \item A set of identities which define lattice with $0$
 \item $x \wedge (x \to y) = x \wedge y$
 \item $x \wedge (y \to z) = x\wedge ((x \wedge y) \to (x \wedge z))$
 \item $z \wedge (x \wedge y \to x) = x$
\end{enumerate}
Note also that if $\mathbb{H}$ is a Heyting algebra,
then (gc2) can be written in the form   
\begin{enumerate}[({gc}2)$^*$]
 \item[({gc}2)$^*$] $a \to \psi (\varphi (a)) = 1$ and $\varphi(\psi(a)) \to a = 1$.
\end{enumerate}

An \emph{HGC-algebra} is an algebra $(H,\vee,\wedge,\to,0,{^\RIGHT},{^\Left})$, where
$\mathbb{H} = (H,\vee,\wedge,\to,0)$ is a Heyting algebra and $({^\RIGHT}, {^\Left})$ is a Galois 
connection on $\mathbb{H}$. By the above, HGC-algebras form an equational class.  
HGC-algebras are usually denoted by $(\mathbb{H},{^\RIGHT},{^\Left})$.
An \emph{H2GC-algebra} $(H,\vee,\wedge,\to,0,{^\RIGHT},{^\Left},{^\Right},{^\LEFT})$ is such
that $\mathbb{H} = (H,\vee,\wedge,\to,0)$ is a Heyting algebra, and  
$(\mathbb{H},{^\RIGHT},{^\Left})$ and $(\mathbb{H},{^\Right},{^\LEFT})$ are
H2GC-algebras, meaning that  $({^\RIGHT}, {^\Left})$ and  $({^\Right}, {^\LEFT})$ are
Galois connections on $\mathbb{H}$. Also H2GC-algebras form
an equational class. We denote H2GC-algebras simply by 
$(\mathbb{H},{^\RIGHT},{^\Left}, {^\Right},{^\LEFT})$.

For an H2GC-algebra  $(\mathbb{H},{^\RIGHT},{^\Left}, {^\Right},{^\LEFT})$, we
introduce the following identities corresponding to the instances (FS1)--(FS4) of
the Fischer Servi axioms:
\begin{enumerate}[({fs}1)]
\item $(a \to b)^\RIGHT \to  (a^\LEFT \to b^\RIGHT) = 1$
\item $(a \to b)^\Right \to  (a^\Left \to b^\Right) = 1$ 
\item $(a^\RIGHT \to b^\LEFT) \to (a \to b)^\LEFT = 1$ 
\item $(a^\Right \to b^\Left) \to (a \to b)^\Left = 1$
\end{enumerate}

In \cite{Dunn94}, J.~M. Dunn  studied minimal 
positive modal logic $\textsf{K}_+$ with the connectives $\wedge$, $\vee$, $\Box$, and
$\Diamond$. $\textsf{K}_+$  can be described in algebraic terms as modal logic based on 
a distributive lattice with two operations $\Box$ and $\Diamond$, where $\Box$ 
distributes over $\wedge$, $\Diamond$ distributes over $\vee$,  and the following 
two conditions hold:\medskip
\begin{enumerate}[($\mathrm{D}_\wedge$)] 
\item[($\mathrm{D}_\wedge$)]  $\Diamond a \wedge \Box b \leq \Diamond( a \wedge b)$
\item[($\mathrm{D}_\vee$)] $\Box(a \vee b) \leq \Box a \vee \Diamond b$
\end{enumerate}
We introduce the instances of ($\mathrm{D}_\wedge$)  as identities defined on an H2GC-algebra:
\begin{enumerate}[({d}1)]
\item $a^\RIGHT \wedge b^\LEFT \to ( a \wedge b)^\RIGHT = 1$
\item $a^\Right \wedge b^\Left \to ( a \wedge b)^\Right = 1$
\end{enumerate}
Now we may write the following proposition.

\begin{proposition} \label{Prop:DunnConditions}
Let $(\mathbb{H},{^\RIGHT},{^\Left}, {^\Right},{^\LEFT})$ be an H2GC-algebra.
\begin{enumerate}[\rm (a)]
\item Identities {\rm (fs1)}, {\rm (d1)}, and {\rm (fs4)} are equivalent.
\item Identities {\rm (fs2)}, {\rm (d2)}, and {\rm (fs3)}  are equivalent.
\end{enumerate}
\end{proposition}

\begin{proof} 
(a) Let us set $a := a \to b$ and $b := a$ in  (d1). We obtain 
$(a \to b)^\RIGHT \wedge  a^\LEFT  \leq (a \wedge (a \to b))^\RIGHT \leq b^\RIGHT$,
because $a \wedge (a \to b) \leq b$. This gives directly 
$(a \to b)^\RIGHT \leq  a^\LEFT  \to b^\RIGHT$, that is,
$(a \to b)^\RIGHT \to  (a^\LEFT  \to b^\RIGHT) = 1$, and so
(d1) implies (fs1). Conversely, if we set $b := a \wedge b$ in (fs1), we have
$b^\RIGHT \leq (a \to b)^\RIGHT \leq (a \to a \wedge b)^\RIGHT \leq a^\LEFT \to (a \wedge b)^\RIGHT$,
because $a \to a \wedge b = a \to b$ and $b \leq a \to b$.
This is equivalent to $a^\LEFT \wedge b^\RIGHT \leq (a \wedge b)^\LEFT$, and
$(a^\LEFT \wedge b^\RIGHT) \to (a \wedge b)^\LEFT = 1$.
Thus, also (fs1) implies (d1), and (fs1) and (d1) are equivalent.

Suppose that (fs1) holds. Then  $(a^\Right \to b^\Left)^\RIGHT \leq  a^{\Right \LEFT} \to b^{\Left \RIGHT}$.
Since $a \leq a^{\Right \LEFT}$ and $b^{\Left \RIGHT} \leq b$, we have  
$(a^\Right \to b^\Left)^\RIGHT \leq  a \to b$. This is equivalent to
$a^\Right \to b^\Left \leq  (a \to b)^\Left$ and
$(a^\Right \to b^\Left) \to  (a \to b)^\Left = 1$,  that is, (fs4) is true.
On the other hand, if (fs4) holds, then
$a \to b \leq a^{\LEFT\Right} \to b^{\RIGHT\Left} \leq (a^\LEFT \to b^\RIGHT)^\Left$,
that is, $(a \to b)^\RIGHT \leq a^\LEFT \to b^\RIGHT$,
$(a \to b)^\RIGHT \leq (a^\LEFT \to b^\RIGHT) = 1$, and (fs1) is true.
Hence, (fs1) and (fs4) are equivalent.

Case (b) can be proved analogously.
\end{proof}

An \emph{H2GC+FS-algebra} is an H2GC-algebra $(\mathbb{H},{^\RIGHT},{^\Left}, {^\Right},{^\LEFT})$
satisfying (fs1) and (fs2). By Proposition~\ref{Prop:DunnConditions}, H2GC+FS-algebras have several
equivalent characterisations. Clearly, H2GC+FS-algebras form an equational class.

\subsection{Rough fuzzy sets on complete Heyting algebras}
\label{sSec:FuzzyMotivation}

We consider here rough lattice-valued fuzzy sets defined on complete Heyting algebras. These
are also closely connected to fuzzy Galois connections studied,
for instance, in \cites{Belo99,GeoPop04}. 

A \emph{complete Heyting algebra} is a Heyting algebra such that its underlying ordered set 
is a complete lattice.  A complete lattice $L$ satisfies the \emph{join-infinite distributive law} 
if for any $S \subseteq L$ and $x \in L$,
\begin{equation*}\label{Eq:JID} \tag{JID}
x \wedge \Big ( \bigvee S \Big ) = \bigvee \{ x \wedge y \mid y \in S \}.
\end{equation*}
A complete lattice is a Heyting algebra if and only if it satisfies \eqref{Eq:JID}
(see e.g. \cites{Grat98,RasSik68}). Thus, complete Heyting algebras are the
complete lattices satisfying \eqref{Eq:JID}.

Fuzzy sets were generalized to $L$-fuzzy sets by J.~A.~Goguen in such a way
that an \emph{$L$-fuzzy set $\varphi$ on $U$} is a mapping 
$\varphi \colon U \to L$, where $U$ is any set representing objects of some universe
of discourse and $L$ is a partially ordered set \cite{Gogu67}. 
The set $L^U$ of all maps from $U$ to $L$ is then the set of all $L$-fuzzy sets on $U$.
The set $L^U$  can be equipped whatever operators $L$ has, and these induced operators obey any law valid 
in $L$ which extends pointwise. 

Here we assume that $\mathbb{H}$ is a complete Heyting algebra, therefore
we can make $H^U$ a complete Heyting algebra by defining
\[
\Big (  \bigvee_{i \in I} \varphi_i \Big ) (a) = \bigvee_{i \in I} \varphi_i(a)  
\quad \text{ and } \quad
\Big (  \bigwedge_{i \in I} \varphi_i \Big ) (a) = \bigwedge_{i \in H} \varphi_i (a)
\]
for all $\{ \varphi_i \}_{i \in I} \subseteq H^U$.
The least element of $H^U$ is $\mathbf{0} \colon x \mapsto 0$
and the greatest element of $H^U$ is $\mathbf{1} \colon x \mapsto 1$.
Furthermore, $H^U$ is relatively pseudocomplemented in such a way that 
for all $\varphi,\psi \in L^U$ and $a \in U$,
\[
  (\varphi \to \psi) (a) = \varphi(a) \to \psi(a)
\]
We denote this complete Heyting algebra by $\mathbb{H}^U$. Elements of this
Heyting algebra are called  \emph{$\mathbb{H}$-sets}.

Dubois and Prade introduced \emph{rough fuzzy sets}  in \cite{DuPr90}. The idea
is that the objects to be approximated are fuzzy  sets, and
the approximations are determined by means of fuzzy relations. Here we study
\emph{rough $\mathbb{H}$-sets}, which means that approximations of $\mathbb{H}$-sets are determined
by $\mathbb{H}$-fuzzy relations. 

Let $\varphi$ be an $\mathbb{H}$-set and let $R$ be an $\mathbb{H}$-fuzzy relation on $U$, 
that is, $R$ is a mapping from $U \times U$ to $H$.
Then, we may define the $\mathbb{H}$-sets $\varphi^\UP$ and $\varphi^\DOWN$ by setting
\begin{eqnarray*}
\varphi^\UP(x) & = & \bigvee_{y \in U} \{ R(x,y) \wedge \varphi(y) \} \\
\varphi^\DOWN(x) & = & \bigwedge_{y \in U} \{  R(x,y) \to \varphi(y) \}
\end{eqnarray*}
for all $x \in U$.
The $\mathbb{H}$-sets $\varphi^\UP$ and $\varphi^\DOWN$ are called the \emph{upper} and
the \emph{lower approximations} of $\varphi$.

We can define another pair of mappings in terms of 
the inverse of $R$ by setting
\begin{eqnarray*}
\varphi^\Up(x) & = & \bigvee_{y \in U} \{ R(y,x) \wedge \varphi(y) \} \\
\varphi^\Down(x) & = & \bigwedge_{y \in U} \{  R(y,x) \to \varphi(y) \}
\end{eqnarray*}
for all $x \in U$.
It is clear that if $\varphi$ is a two-valued set on $U$ and $R$
is a two-valued binary relation on $U$, then the operations $^\UP$, $^\DOWN$, $^\Up$, and $^\Down$
coincide with the rough set operators defined by a binary relation.

\begin{proposition}\label{Prop:RoughFuzzy}
For any complete Heyting algebra $\mathbb{H}$ and an $\mathbb{H}$-fuzzy relation $R$ on $U$, the algebra of rough 
$\mathbb{H}$-sets $(\mathbb{H}^U, {^\UP}, {^\DOWN}, {^\Up}, {^\Down})$ is 
an {H2GC+FS}-algebra.
\end{proposition}

\begin{proof}
Suppose $\varphi$ and $\psi$ are $\mathbb{H}$-sets such that
$\varphi \leq \psi$. Then for all $y \in U$, $R(x,y) \wedge \varphi(y) \ \leq  \ R(x,y) \wedge \psi(y)$
and this implies
\[ \varphi^\UP(x) = \bigvee_{y \in U}  \{ R(x,y) \wedge \varphi(y) \} \leq  
\bigvee_{y \in U}  \{ R(x,y) \wedge \psi(y) \} = \psi^\UP(x). \]
Similarly, $ R(y,x) \to \varphi(y) \ \leq \ R(y,x) \to \psi(y) $ for all 
$y \in U$. Thus,
\[ \varphi^\Down(x) =  \bigwedge_{y \in U} \{  R(y,x) \to \varphi(y) \} \leq
\bigwedge_{y \in U} \{  R(y,x) \to \psi(y) \} =  \psi^\Down(x). \]
So, $^\UP$ and $^\Down$ are order-preserving.
By definition, for all $x \in U$, 
\begin{eqnarray*}
\varphi^{\Down \UP}(x) & = & \bigvee_{y \in U}  \{ R(x,y) \wedge \varphi^\Down (y)  \}  
=  \bigvee_{y \in U} \Big \{ R(x,y) \wedge  \bigwedge_{z \in U} \{  R(z,y) \to \varphi(z) \} \, \Big \} \\
& \leq & \bigvee_{y \in U} \{ R(x,y) \wedge  ( R(x,y) \to \varphi(x) \, ) \, \} 
 \leq  \bigvee_{y \in U} \{  \varphi(x) \} =  \varphi(x).
\end{eqnarray*}
This means that $\varphi^{\UP \Down} \leq \varphi$. 
Analogously, for any $x \in U$, 
\begin{eqnarray*}
\varphi^{\UP \Down}(x) & = & \bigwedge_{y \in U} \{  R(y,x) \to \varphi^\UP (y) \} 
=  \bigwedge_{y \in U} \Big \{  R(y,x) \to \bigvee_{z \in U} \{ R(y,z) \wedge \varphi(z) \} \, \Big \} \\
& \geq & \bigwedge_{y \in U} \{  R(y,x) \to ( \, R(y,x) \wedge \varphi(x) \, ) \} 
\geq  \bigwedge_{y \in U} \{  \varphi(x) \} =  \varphi(x). 
\end{eqnarray*}
Thus, also $\varphi \leq \varphi^{\UP \Down}$. We have that  $({^\UP},{^\Down})$
is a Galois connection, because (gc1) and (gc2) are satisfied. Similarly, 
we can show that $({^\Up},{^\DOWN})$ is a Galois connection.

Next we show that (d1) holds. For all $x,y \in U$, we have
\begin{align*}
 R(x,y) \wedge \varphi(y)  \wedge \psi(x)^\DOWN &=
 R(x,y) \wedge \varphi(y)  \wedge \bigwedge_{z \in U} \{  R(x,z) \to \psi (z) \}\\
 & \leq R(x,y) \wedge \varphi(y) \wedge (R(x,y) \to \psi(y)) \\
 & = (R(x,y) \wedge (R(x,y) \to \psi(y)))  \wedge \varphi(y) \\
 & = R(x,y) \wedge \psi(y)  \wedge \varphi(y) \\
 & = R(x,y) \wedge (\varphi \wedge \psi)(y)\\
 & \leq \bigvee_{z \in U} \{ R(x,z) \wedge (\varphi \wedge \psi)(z) \}\\
 & = (\varphi \wedge \psi)^\UP(x).
\end{align*}
Hence, for all $y \in U$, 
\[ R(x,y) \wedge \varphi(y) \wedge \psi(x)^\DOWN \leq  (\varphi \wedge \psi)^\UP(x). \]
Because complete Heyting algebras satisfy the join-infinite distributive law, we have that
for all $x \in U$,
\begin{align*}
(\varphi^\UP \wedge \psi^\DOWN)(x) & = \varphi^\UP(x) \wedge \psi^\DOWN(x)  
= \bigvee_{y \in U} \{ R(x,y) \wedge \varphi(y)\} \wedge \psi^\DOWN(x) \\
& = \bigvee_{y \in U} \{ R(x,y) \wedge \varphi(y) \wedge \psi^\DOWN(x) \} 
\leq  (\varphi \wedge \psi)^\UP(x). 
\end{align*}
Thus, $\varphi^\UP \wedge \psi^\DOWN \leq  (\varphi \wedge \psi)^\UP$.
Assertion (d2) can be proved similarly.
\end{proof}

\begin{example}\label{Exa:NotDunn2}
The instances
\[
 (a \vee b)^\LEFT \leq a^\LEFT \vee b^\RIGHT \quad \text{ and } \quad (a \vee b)^\Left \leq a^\Left \vee b^\Right
\]
of  Dunn's axiom (D$_\vee$) are false in some H2GC+FS-algebras of rough $\mathbb{H}$-sets.

Namely, let $U = \{x,y\}$ and consider the finite (and hence complete) Heyting algebra $\mathbf{2}^2 \oplus 1$, 
that is, $\mathbb{H} = \{0, a, b, c, 1\}$ the Heyting algebra with the order 
$0 < a, b < c < 1$, where $a$ and $b$ are incomparable. Note that $\neg a = b$ and $\neg b = a$.

We define two  $\mathbb{H}$-sets $\varphi, \psi$ on $U$ by setting $\varphi(u) = 0 $ and $ \psi(u) = 1$ for all $u \in U$. 
An $\mathbb{H}$-fuzzy relation $R \colon U \times U \to H$ is defined by  
$R(x, x) = R(y,y) = a$ and $R(x, y) = R(y,x) = b$. Then,
\begin{align*} 
(\varphi \vee \psi)^\DOWN (x) &= \bigwedge_{u \in U}  (R(x,u) \to (\varphi \vee \psi) (u) 
     = \bigwedge_{u \in U}  (R(x,u) \to (\varphi(u) \vee \psi(u)) \\
     &= \bigwedge_{u \in U}  (R(x,u) \to 1) 
     = (a \to 1) \wedge (b \to 1) 
     = 1 \wedge 1 = 1,
\end{align*}
but 
\begin{align*}
\varphi ^\DOWN (x) \vee  \psi^\UP (x) &= \bigwedge_{u \in U}  (R(x,u) \to  \varphi (u)) \vee \bigvee_{u \in U} (R(x,u) \wedge \psi (u)) \\   
&= \bigwedge_{u \in U}  (R(x,u) \to 0) \vee \bigvee_{u \in U} (R(x,u) \wedge 1) \\   
&= \bigwedge_{u \in U}  \neg R(x,u) \vee \bigvee_{u \in U} R(x,u) \\   
&= ( \neg R(x,x) \wedge \neg R(x,y)  ) \vee ( R(x,x) \vee R(x,y) ) \\
&= ( \neg a \wedge \neg b  ) \vee ( a \vee b ) 
= 0 \vee c = c.
\end{align*}
Hence condition $(\varphi \vee \psi)^\LEFT \leq \varphi^\LEFT \vee \psi^\RIGHT$ is not satisfied, because $1 \nleq c$. 
Similarly, we can show that 
\begin{center}
$(\varphi \vee \psi)^\Down (y) = 1$ \ and \ $\varphi^\Down (y) \vee  \psi^\Up (y) = 0 \vee c = c$ ,
\end{center}
that is, $(\varphi \vee \psi)^\Left \leq \varphi^\Left \vee \psi^\Right$ is not satisfied.
\end{example}

Hence, we may conclude this subsection by stating that the rough lattice-valued fuzzy sets defined on complete Heyting algebras 
are algebraic models for \textsf{Int2GC+FS}.

\subsection{Algebraic Semantics and Completeness}
\label{sSec:AlgCompleteness}

As we already noted, \textsf{IntGC} is complete with respect to HGC-algebras. 
Here we show completeness of \textsf{Int2GC} and \textsf{IntGC+FS} with respect to
H2GC- and H2GC+FS-algebras.

Let $(\mathbb{H}, {^\RIGHT}, {^\LEFT}, {^\Right}, {^\Left})$ be an H2GC-algebra,
where $\mathbb{H} = (H,\vee,\wedge,\to,0)$. A \emph{valuation} is a function $v \colon \mathrm{Var} \to H$ 
assigning to each propositional variable $p$ an element $v(p)$ of $H$. 
Let $\Phi$ denote the set of well-formed {\sf Int2GC}-formulas. Clearly, $\Phi$
is the set of well-formed {\sf Int2GC+FS}-formulas as well, because these logics have the same language.
The valuation $v$ can be extended to the set $\Phi$ inductively\footnote{%
Note that the idea is that the operations ${^\RIGHT}$, ${^\LEFT}$, ${^\Right}$, ${^\Left}$ may be obtained
from their logical counterparts ${\UP}$, ${\DOWN}$, ${\Up}$, ${\Down}$ just by turning 
them 90 degrees clockwise.}:
\begin{align*}
 v(\neg A)    &= v(A) \to 0       & v(A \to B)   &= v(A) \to v(B) \\
 v(A\wedge B) &= v(A) \wedge v(B) & v(A \vee B) &= v(A) \vee v(B) \\
 v(\UP A)     &= v(A)^\RIGHT       & v(\DOWN A) &=  v(A)^\LEFT\\
 v(\Up A)     &= v(A)^\Right       & v(\Down A) &= v(A)^\Left
\end{align*}
An \textsf{Int2GC}-formula $A$ is \emph{valid} if $v(A) = 1$ for any valuation $v$
on any H2GC-algebra. Similarly, we may define validity of \textsf{Int2GC+FS}-formulas 
over H2GC+FS-algebras.

\begin{theorem}[\bf Soundness I] \label{Thm:SoundnessI}
\text{ }
\begin{enumerate}[\rm (a)]
 \item Provable {\sf Int2GC}-formulas are valid in H2GC-algebras
 \item Provable {\sf Int2GC+FS}-formulas are valid in H2GC+FS-algebras.
\end{enumerate}
\end{theorem}

\begin{proof} The proof concerning intuitionistic logic is standard 
(see \cite{RasSik68}, for instance). As we have proved in \cite{DzJaKo10} for
{\sf IntGC}, rules (GC\,${\Down}{\UP}$), (GC\,${\UP}{\Down}$), 
(GC\,${\DOWN}{\Up}$), (GC\,${\Up}{\DOWN}$) preserve validity. Thus, (a) holds. 
For (b), it is clear that axioms (FS1) and (FS2) are valid, because H2GC+FS-algebras
satisfy identities (fs1) and (fs2).
\end{proof}

To obtain completeness, we  apply Lindenbaum--Tarski algebras. 
We denote by $\mathbf{\Phi}$ the \emph{algebra of $\Phi$-formulas}, that is, 
the abstract algebra 
\[ \mathbf{\Phi} = (\Phi,\vee,\wedge,\to,\bot,{\UP},{\DOWN},{\Up}, {\Down}).\]
We define two equivalences $\equiv_1$ and $\equiv_2$ on $\Phi$:
\begin{align*}
A \equiv_1 B & \text{ if and only if } A \leftrightarrow B \text{ is provable in {\sf Int2GC}};\\
A \equiv_2 B & \text{ if and only if } A \leftrightarrow B \text{ is provable in {\sf Int2GC+FS}}.
\end{align*}

Concerning $\vee$, $\wedge$, and $\to$, the next result is known from the theory of
intuitionistic logic, and for  ${\UP}$, ${\DOWN}$, ${\Up}$, and ${\Down}$ the claim follows 
from monotonicity.

\begin{lemma}\label{Lem:Congruence}
The equivalences $\equiv_1$ and $\equiv_2$ are congruences on $\mathbf{\Phi}$.
\end{lemma}

For any $A \in \Phi$, we denote by $[A]_1$ and $[A]_2$ the congruence class of $A$
with respect to the congruences $\equiv_1$ and $\equiv_2$. The sets of $\equiv_1$- and
$\equiv_2$-classes are denoted by $\Phi/\!\!\equiv_1$ and $\Phi/\!\!\equiv_2$.
Next we define the \emph{quotient algebras} of $\mathbf{\Phi}$ with respect
to $\equiv_1$ and $\equiv_2$ by introducing the following operations on 
$\Phi/\!\!\equiv_i$ for $i = 1,2$:
\begin{align*}
{[A]_i} \vee_i [B]_i &=  [A \vee B]_i         &{[A]_i} \wedge_i [B]_i  &=  [A \wedge B]_i \\
{[A]_i} \to_i [B]_i  &=  [A \to B]_i          &\mathbf{0}_i     &= [\perp]_i  \\
{[A]_i}^{\RIGHT_i}      &=  [\UP A]_i            &{[A]_i}^{\LEFT_i} & = [\DOWN A]_i \\
{[A]_i}^{\Right_i}      &=  [\Up A]_i            &{[A]_i}^{\Left_i} & = [\Down A]_i 
\end{align*}

As we have noted, H2GC- and H2GC+FS-algebras form equational classes. By 
the theory of intuitionistic logic, $\Phi/{\equiv_1}$ satisfies the identities defining
Heyting algebras. Also (gc1) and (gc2) hold by (GC1), (GC1)$^*$, (GC4), (GC4)$^*$. 
Since also identities (fs1) and (fs2) are the counterparts of axioms (FS1) and (FS2), we
may write the following proposition.

\begin{proposition}\label{Prop:HLG-algebra}
\text{ }
\begin{enumerate}[\rm (a)]
 \item The algebra $(\Phi/{\equiv_1}, \vee_1, \wedge_1, \to_1, \mathbf{0}_1, {^{\RIGHT_1}}, {^{\LEFT_1}}, {^{\Right_1}}, 
       {^{\Left_1}})$ is an H2GC-algebra.
 \item The algebra $(\Phi/{\equiv_2}, \vee_2, \wedge_2, \to_2, \mathbf{0}_2, {^{\RIGHT_2}}, {^{\LEFT_2}}, {^{\Right_2}}, 
       {^{\Left_2}})$ is an H2GC+FS-algebra.
\end{enumerate}
\end{proposition}

We define two valuations $v_1 \colon \mathit{Var} \to \Phi /\!\!\equiv_1$ and  
$v_2 \colon \mathit{Var}  \to \Phi /\!\!\equiv_2$ by:
\[
   v_1(p) = [p]_1 \text{ \ and \ }  v_2(p) = [p]_2.
\]
By a straightforward formula induction we see that $v_1(A) = [A]_1$ and  $v_2(A) = [A]_2$ for all formulas $A \in \Phi$.
We can now write the following results.

\begin{lemma} \label{Lem:Truth}
For any formula $A \in \Phi$:
\begin{enumerate}[\rm (a)]
 \item $A$ is provable in {\sf Int2GC} if and only if $v_1(A) = \mathbf{1}$.
 \item $A$ is provable in {\sf Int2GC+FS} if and only if $v_2(A) = \mathbf{1}$.
\end{enumerate}
\end{lemma}

\begin{theorem}[\bf Completeness~I] \label{Thm:CompletenessI}
For any formula $A \in \Phi$:
\begin{enumerate}[\rm (a)]
 \item $A$ is provable in {\sf Int2GC} if and only if $A$ valid in H2GC-algebras.
 \item $A$ is provable in {\sf Int2GC+FS} if and only if $A$ valid in H2GC+FS-algebras.
\end{enumerate}
\end{theorem}

\begin{proof}
(a) Suppose that $A$ is valid in H2GC-algebras. We have $v_1(A) = \mathbf{1}$ in $\Phi/{\equiv_1}$, 
that is, $A$ is provable. The other direction is proved in Theorem~\ref{Thm:SoundnessI}. For (b),
the proof is basically the same.
\end{proof}

Clearly, rough $\mathbb{H}$-sets considered in Section~\ref{sSec:FuzzyMotivation} 
provide algebraic models for Int2GC+FS.
Let us introduce axioms (D1) and (D2) corresponding to equations (d1) and (d2):
\begin{enumerate}[({D}1)]
 \item $\UP A \wedge \DOWN B \to \UP(A \wedge B)$ 
 \item $\Up A \wedge \Down B \to \Up(A \wedge B)$
\end{enumerate}
By  Proposition~\ref{Prop:DunnConditions} and Theorem~\ref{Thm:CompletenessI}, we can write the following 
corollary giving additional ways to axiomatize {\sf Int2GC+FS}.
\begin{corollary} \label{Cor:Axiomatisation} 
\[ \mbox{\sf Int2GC+FS} = {\sf Int2GC} + \{ {\rm (FS1)} \ {\rm or} \ {\rm (FS4)} \ {\rm or} \ {\rm (D1)} \}
+ \{ {\rm (FS2)}  \ {\rm or} \ {\rm (FS3)} \ {\rm or} \ {\rm (D2)} \} \]
\end{corollary}
By Example~\ref{Exa:NotDunn2} and Theorem~\ref{Thm:CompletenessI}, 
formulas $\DOWN(A \vee B) \to \DOWN A \vee \UP B$ and $\Down(A \vee B) \to \Down A \vee \Up B$ 
corresponding to Dunn's condition (D$_\vee$) are not provable in {\sf Int2GC+FS}.

\section{Kripke semantics and completeness}
\label{Sec:Kripke}

\subsection{Kripke frames}
\label{sSec:Frames}

In this section we consider Kripke frames and models for the three systems {\sf IntGC}, {\sf Int2GC},
and {\sf Int2GC+FS}, where the system {\sf Int2GC+FS} will be shown in Section~\ref{Sec:IKt} to be equivalent 
to intuitionistic temporal logic {\sf IK$_t$}. 

\medskip\noindent%
{\bf IntGC-frames.} \ A structure $\mathcal{F} = (X, \le, R)$ is called a \emph{Kripke frame of {\sf IntGC}}
(an {\sf IntGC}-{\em frame}, in short) \cite{DzJaKo10}, if $X$ is a non-empty set, 
$\le$ is a preorder (reflexive and transitive binary relation) on $X$, and $R$ is a 
relation on $X$ such that
\begin{equation} \tag{R1} \label{Eq:frame} 
({\ge} \circ R \circ {\ge}) \subseteq R .
\end{equation}

Let $v$ be a function $v \colon P \to \wp(X)$ assigning to each propositional variable $p$ 
a subset $v(p)$ of $X$ with the property that $x\in v(p)$ and $x\le y$ imply $y\in v(p)$,
that is, $v(p)$ is $\le$-closed. Such functions are called \emph{valuations} and the pair 
$\mathcal{M} = (\mathcal{F}, v)$ is called an {\sf IntGC}-\emph{model}.
For any $x \in X$ and $A \in \Phi$, we define the \emph{satisfiability relation} in
$\mathcal{M}$ inductively by the following way:
\begin{align*}
x \models p &\iff x \in v(p), \\
x \models A \wedge B &\iff x \models A \mbox{ and } x \models A, \\
x \models A \vee B &\iff x \models A \mbox{ or } x \models A, \\
x \models A \to B &\iff \mbox{ for all } y \geq x, \  y \models A \mbox{ implies }  y \models B, \\
x \models \neg A &\iff \mbox{ for no }y \geq x \mbox{ does }  y  \models A, \\
x \models \UP A &\iff \mbox{ exists } y \mbox{ such that } x \, R \, y \mbox{ and }  y \models  A, \mbox{ and}\\
x \models \Down A &\iff \mbox{ for all } y, y \, R \, x \mbox{ implies }  y \models  A.
\end{align*}
Note that the satisfiability relation $\models$ is \emph{persistent}, that is, for all
formulas $A$, if $x \models A$ and $x \leq y$, then $y \models A$. 
An {\sf IntGC}-formula $A$ is \emph{valid in a model} $\mathcal{M}$, if $x \models A$ for all $x \in X$. 
The formula $A$ is \emph{valid in a frame} $\mathcal{F}$, if $A$ is valid in every model based on $\mathcal{F}$. 
A formula is \emph{Kripke-valid} if it is valid in every frame.

We noted in \cite{DzJaKo10} that {\sf IntGC} is \emph{Kripke-sound}, that is, every provable
{\sf IntGC}-formula is Kripke-valid. We also proved Kripke-completeness by applying canonical
frames, and next we shortly recall these constructions, because a similar technique will be
used later in cases of {\sf Int2GC} and {\sf Int2GC+FS}.

For an HGC-algebra $(\mathbb{H},{^\RIGHT},{^\Left})$, its \emph{canonical frame} is a
triple $(X_H, \subseteq, R^H)$ such that $X_H$ is the set of the prime filters of the
lattice $H$ and the relation $R^H$ is defined by
\[ (x,y) \in R^H \iff  y \subseteq [x]^{\RIGHT^{-1}},\]
where $[x]^{\RIGHT^{-1}} = \{ a \in H \mid a^\RIGHT \in x\}$. The relation $R^H$
can be described also in terms of the map $^\Left$ by
\[ (x,y) \in R^H \iff [y]^{\Left^{-1}} \subseteq x ,\]
where $[y]^{\Left^{-1}} = \{ a \in H \mid a^\Left \in y\}$.

For a Heyting algebra $\mathbb{H}$, we denote by $\mathcal{O}(H)$ the set of all $\le$-closed subsets. 

\begin{lemma} \label{Lem:InducedGalois}
Let $(\mathbb{H},{^\RIGHT},{^\Left})$ be an HGC-algebra.
The pair $\big ( {^{\Left^{-1}}},{^{\RIGHT^{-1}}} \big )$ is a Galois connection on $(\mathcal{O}(H),\subseteq)$.
\end{lemma}

\begin{proof} Let $x,y \in \mathcal{O}(H)$. Suppose $[x]^{\Left^{-1}} \subseteq y$. If $a \in x$, then $a \leq a^{\RIGHT \Left}$ 
implies $a^{\RIGHT \Left} \in x$, because $x \in \mathcal{O}(H)$. 
This means that $a^\RIGHT \in [x]^{\Left^{-1}} \subseteq y$ and so $a \in [y]^{\RIGHT^{-1}}$.
Therefore, $x \subseteq {[y]^{\RIGHT^{-1}}}$.
Conversely, assume that $x \subseteq [y]^{\RIGHT^{-1}}$. If $a \in [x]^{\Left^{-1}}$, then 
$a^\Left \in x \subseteq [y]^{\RIGHT^{-1}}$, that is, $a^{\Left \RIGHT} \in y$. Because $a^{\Left \RIGHT} \leq a$,
we have $a \in y$, since $y \in \mathcal{O}(H)$. Thus, $[x]^{\LEFT^{-1}} \subseteq y$.
\end{proof}

Let $(\mathbb{H},{^\RIGHT},{^\Left})$ be an HGC-algebra. In \cite{DzJaKo10}, we showed that the canonical
frame $\mathcal{F}^H = (X_H, \subseteq, R^H)$ is an {\sf IntGC}-frame. Let $v \colon \mathrm{Var} \to H$
be a valuation on this HGC-algebra. We may now define a valuation $v^* \colon \mathrm{Var} \to X_H$
for the canonical frame $\mathcal{F}^H$ by setting $x \in v^*(p)$ if and only if $v(p) \in x$ for all
$p \in \mathrm{Var}$. Obviously, for all $x,y \in X_H$ and $p \in \mathrm{Var}$, 
$x \in v^*(p)$ and $x \subseteq y$ imply $y \in v^*(p)$, so $v^*$ is really a valuation. 
In the \emph{canonical model} $(\mathcal{F}^H,v^*)$, we have 
$x \models p$ if and only if $v(p) \in x$ for all $x \in X_H$. In \cite{DzJaKo10}, we proved by formula induction the \emph{Key Lemma}
stating that for any {\sf IntGC}-formula $A$ and $x \in X_H$, $x \models A$ if and only if $v(A) \in x$. 
This enabled us to prove the Kripke-completeness, that is, an {\sf IntGC}-formula is provable if and only 
if it is Kripke-valid. 

\medskip\noindent%
{\bf Int2GC-frames.} \  An {\sf Int2GC}-{\em frame}
(or a \emph{Kripke frame of\/ {\sf Int2GC}}) is a quadruple $\mathcal{F} = (X, \le, R_1,R_2)$ such that $X$ is a non-empty set, 
$\le$ is a preorder on $X$, and $R_1$ and $R_2$ are relations on $X$ satisfying
\begin{equation} \tag{R2} \label{Eq:frame2} 
({\ge} \circ R_1 \circ {\ge}) \subseteq R_1 
\end{equation}
\begin{equation} \tag{R3} \label{Eq:frame3} 
({\le} \circ R_2 \circ {\le}) \subseteq R_2 .
\end{equation}
Our next lemma is obvious.

\begin{lemma} \label{Lem:FrameConnection1}
$(X,\leq,R_1,R_2)$ is an {\sf Int2GC}-frame \ if and only if \ \ $(X,\leq,R_1)$ and $(X,\leq,{R_2}^{-1})$ are
{\sf IntGC}-frames. 
\end{lemma}

In {\sf Int2GC}-frames the valuations and the satisfiability relation $\models$ for $\vee$, $\wedge$, $\to$, and $\neg$ 
are defined as earlier, but satisfiability of formulas $\UP A$, $\Down A$, $\Up A$, and $\DOWN A$ are defined by 
\begin{align*}
x \models \UP A &\iff \mbox{ exists } y \mbox{ such that } x \, R_1 \, y \mbox{ and }  y \models  A, \\
x \models \Down A &\iff \mbox{ for all } y, y \, R_1 \, x \mbox{ implies }  y \models  A, \\
x \models \Up A &\iff \mbox{ exists } y \mbox{ such that } y \, R_2 \, x \mbox{ and }  y \models  A, \mbox{ and} \\
x \models \DOWN A &\iff \mbox{ for all } y, x \, R_2 \, y \mbox{ implies }  y \models  A.
\end{align*}
It is obvious that {\sf Int2GC} is \emph{Kripke-sound}, that is, every formula provable in
{\sf Int2GC} is Kripke valid.

We can introduce two \textsf{IntGC}-logics, one with the operators $\UP$ and $\Down$, and the other
with $\Up$ and $\DOWN$. We denote these by \textsf{IntGC$_1$} and \textsf{IntGC$_2$}, respectively.
Next we show that {\sf Int2GC} extends \textsf{IntGC$_1$} and \textsf{IntGC$_2$}.

\begin{lemma} Let $\mathcal{F}_i = (X,\leq,R_i)$ be a Kripke-frame for {\sf IntGC$_i$} and let 
$A_i$ be a well-formed formula of {\sf IntGC$_i$}, where $i = 1,2$.
Then, $A_i$ is valid in $\mathcal{F}_i$ if and only if $A_i$ is valid in $\mathcal{F} = (X,\leq,R_1,R_2^{-1})$. 
\end{lemma}

\begin{proof} We prove the claim by formula induction. Concerning {\sf IntGC$_1$}-frames and -formulas,
the claim is obvious.

Let $v$ be a valuation for the frame  $\mathcal{F}_2 = (X,\leq,R_2)$. Thus,  
$\mathcal{M}_2 = (\mathcal{F}_2,v)$ is an {\sf IntGC$_2$}-model and 
$\mathcal{M} = (\mathcal{F},v)$ is an {\sf Int2GC}-model.

Let $A_2$ be a formula of {\sf IntGC$_2$} which is of the form $\Up A$ for some
{\sf IntGC$_2$}-formula $A$ having this property. Then, for all $x \in X$,
\begin{align*}
\mathcal{M}_2,x \models \Up A & \iff  (\exists y) \, x \, R_2 \, y \text{ and } \mathcal{M}_2,y \models A \\
& \iff  (\exists y) \, x \, R_2 \, y \text{ and } \mathcal{M},y \models A \\
& \iff  (\exists y) \, y \, {R_2}^{-1} \, x \text{ and } \mathcal{M},y \models A \\
& \iff \mathcal{M},x \models \Up A
\end{align*}
Therefore, $A_2$ is valid in $\mathcal{F}_2$ if and only if $A_2$ is valid in $\mathcal{F}$. The
claim concerning the operator $\DOWN$ can be proved analogously. 
\end{proof}

The \emph{canonical {\sf Int2GC}-frame} of an H2GC-algebra $(\mathbb{H},{^\RIGHT},{^\Left}, {^\Right},{^\LEFT})$
is a structure $(X_H,\subseteq, R_1^H, R_2^H)$, where $X_H$ is the set of lattice-filters of $H$ and the
relations $R_1^H$ and $R_2^H$ are defined by 
\[ (x,y) \in R_1^H \iff  y \subseteq [x]^{\RIGHT^{-1}} \text{ \quad and \quad }
  (x,y) \in R_2^H \iff  [x]^{\LEFT^{-1}}  \subseteq y
\]
where $[x]^{\LEFT^{-1}} = \{ a \in H \mid a^\LEFT \in x\}$. Equivalently, these relations can be defined as
\[ (x,y) \in R_1^H \iff [y]^{\Left^{-1}} \subseteq x \text{ \quad and \quad }
  (x,y) \in R_2^H \iff  x \subseteq [y]^{\Right^{-1}} 
\]
in which $[y]^{\Right^{-1}} = \{ a \in H \mid a^\Right \in y\}$.
The next lemma is obvious and its proof is omitted.

\begin{lemma} \label{Lem:InducedGalois2}
$\big ( {^{\Left^{-1}}},{^{\RIGHT^{-1}}} \big )$ and
$\big ( {^{\LEFT^{-1}}},{^{\Right^{-1}}} \big )$
are Galois connections on $(\mathcal{O}(H),\subseteq)$.
\end{lemma}

Similarly, as in the case of HGC-algebras, we can show that the canonical frame  
$\mathcal{F}^H = (X_H,\subseteq, R_1^H, R_2^H)$ of any
H2GC-algebra $(\mathbb{H},{^\RIGHT},{^\Left}, {^\Right},{^\LEFT})$ is an {\sf Int2GC}-frame. 
For any valuation $v$ on $H$, we can define the valuation $v^*$ for the canonical frame  $\mathcal{F}^H$ 
by setting $x \in v^*(p)$ if and only if $v(p) \in x$ for all propositional variables $p \in \mathrm{Var}$ 
and $x \in X_H$. As in case of {\sf IntGC} (see Lemma~5.7 in \cite{DzJaKo10}), we can prove by 
formula induction that the Key Lemma holds, that is, for any {\sf Int2GC}-formula $A$ and 
$x \in X_H$, $x \models A$ if and only if $v(A) \in x$. 
Therefore, we may state the Kripke-completeness presented in the next theorem.

\begin{theorem}[\bf Completeness~II for Int2GC] \label{Thm:Completeness2}
A formula $A$ is provable in {\sf Int2GC} if and only if $A$ is Kripke-valid.
\end{theorem}

\begin{proof} Suppose that an {\sf Int2GC}-formula $A$ is not provable. This means that
there exists an H2GC-algebra  $(\mathbb{H},{^\RIGHT},{^\Left}, {^\Right},{^\LEFT})$ 
and a valuation $v \colon \textrm{Var} \to H$ such that $v(A) \neq 1$. We construct the 
canonical frame $\mathcal{F}^H$ and the valuation $v^*$ as above. Because $v(A) \neq 1$,
there exists a prime filter $x$ such that $v(A) \notin x$. By the Key Lemma, this means that 
$x \not \models A$ in the canonical model $(\mathcal{F}^H,v^*)$. 
Therefore, $A$ is not Kripke-valid.
\end{proof}

\medskip\noindent%
{\bf Int2GC+FS-frames.} \ An {\sf Int2GC+FS}-{\em frame} (of a \emph{Kripke frame of\/ {\sf Int2GC+FS}}) 
is a tripe $\mathcal{F} = (X, \le, R)$, where $X$ is a non-empty set, $\le$ is a preorder on $X$, and $R$ is a 
relation on $X$ satisfying
\begin{equation} \tag{R4} \label{Eq:frame4} 
(R \circ {\leq}) \subseteq ({\leq} \circ R)
\end{equation}
\begin{equation} \tag{R5} \label{Eq:frame5} 
({\geq} \circ R) \subseteq (R \circ {\geq}).
\end{equation}

\begin{lemma} \label{Lem:FrameConnection2}
$(X,\leq,R)$ is an {\sf Int2GC+FS}-frame if and only if $(X,\leq,R \circ {\geq}, {\leq} \circ R)$
is an {\sf Int2GC}-frame.
\end{lemma}

\begin{proof}
Suppose $(X,\leq,R)$ is an {\sf Int2GC+FS}-frame. Let $R_1 = R \circ {\geq}$ and $R_2 = {\leq} \circ R$.
We show that the relations $R_1$ and $R_2$ satisfy conditions \eqref{Eq:frame2} and \eqref{Eq:frame3}.
Now ${\geq} \circ R_1 \circ {\geq} = {\geq} \circ (R \circ {\geq}) \circ {\geq} = {\geq} \circ R \circ {\geq} = 
({\geq} \circ R) \circ {\geq} \subseteq  (R \circ {\geq}) \circ {\geq} = R \circ {\geq} = R_1$, that is,
\eqref{Eq:frame2} is satisfied. Similarly, 
${\leq} \circ R_2 \circ {\leq} =  {\leq} \circ ({\leq} \circ R) \circ {\leq} = {\leq} \circ R \circ {\leq}
= {\leq} \circ (R \circ {\leq}) \subseteq {\leq} \circ ({\leq}  \circ R) = {\leq} \circ R = R_2$, and also
\eqref{Eq:frame3} holds. Thus, $(X,\leq,R \circ {\geq}, {\leq} \circ R)$ is an {\sf Int2GC}-frame.

Conversely, suppose that $(X,\leq,R \circ {\geq}, {\leq} \circ R)$ is an {\sf Int2GC}-frame. We again
put $R_1 = R \circ {\geq}$ and $R_2 = {\leq} \circ R$. Then, $R \circ {\leq} \subseteq
{\leq} \circ  R \circ {\leq} = R_2 \circ {\leq} \subseteq {\leq} \circ R_2 \circ {\leq} \subseteq
R_2 = {\leq} \circ R$, that is, \eqref{Eq:frame4} holds. Similarly, ${\geq} \circ R \subseteq
{\geq} \circ R \circ {\geq} = {\geq} \circ R_1 \subseteq {\geq} \circ R_1 {\geq} \subseteq
R_1 = R \circ {\geq}$. Hence, also \eqref{Eq:frame5} is satisfied and  $(X,\leq,R)$ is an {\sf Int2GC+FS}-frame.
\end{proof}

\begin{corollary}
$(X,\leq,R)$ is an {\sf Int2GC+FS}-frame if and only if $(X,\leq,R \circ {\geq})$ and 
$(X,\leq, R^{-1} \circ {\geq})$ are {\sf IntGC}-frames
\end{corollary}

\begin{proof}
The claim follows directly from Lemmas~\ref{Lem:FrameConnection1} and \ref{Lem:FrameConnection2},
because $({\leq} \circ R)^{-1} = R^{-1} \circ {\geq}$.
\end{proof}

Again, in {{\sf Int2GC+FS}-frames} the valuations and the satisfiability relation $\models$ for $\vee$, $\wedge$, $\to$, 
and $\neg$ are defined as earlier, and satisfiability of $\UP A$, $\Down A$, $\Up A$, and $\DOWN A$ are defined by
\begin{align*}
x \models \UP A   &\iff \mbox{ exists } y \mbox{ such that } x \, (R \circ {\geq}) \, y \text{ and } y \models  A \\
x \models \Down A &\iff \mbox{ for all } y, y \, (R \circ {\geq}) \, x \text{ implies } y \models  A \\
x \models \Up A   &\iff \mbox{ exists } y \mbox{ such that } y \, ({\leq} \circ R) \, x \text{ and } y \models  A \\
x \models \DOWN A &\iff \mbox{ for all } y, x \, ({\leq} \circ R) \, y \text{ implies } y \models  A 
\end{align*}

\begin{lemma} \label{Lem:Perisisent}
For all {\sf Int2GC+FS}-models $\mathcal{M} = (\mathcal{F},v)$ and formulas $A \in \Phi$: 
\[  x \models  A \mbox{ \ and \ } x \leq y \mbox{ \ imply \ }  y \models  A.\]
\end{lemma}

\begin{proof} As an example, we show the claim for $\Up$ and $\DOWN$.

(${\Up}A$) Suppose $x \models \Up A$ and $x \leq y$. Then, there exists $z$ such that
$z({\leq}\circ R)x$ and $z \models A$. Thus, there is $w$ such
that $z \leq w$, $wRx$, and $z \models A$. Now $wRx$ and $x \leq y$
imply $w(R \circ {\leq})y$. From frame condition \eqref{Eq:frame4}, we
get $w({\leq} \circ R)y$. Now $z \leq w$ implies $z({\leq} \circ R)y$.
Since $z \models A$, we have $y \models \Up A$.

(${\DOWN} A$) Assume that $x \models \DOWN A$, $x \leq y$, but $y \not \models \DOWN A$.
Then, there exists $z$ such that $y({\leq} \circ R)z$ and $z \not \models A$. Since $x \leq y$, we have 
$x({\leq} \circ R)z$. By $z \not \models A$,
we get  $x \not \models \DOWN A$, a contradiction.
\end{proof}

Our next lemma showing a connection between validity in {\sf Int2GC+FS}-frames and {\sf Int2GC}-frames is
obvious and thus its proof is omitted.
\begin{lemma} \label{Lem:Validity}
Let $A \in \Phi$. Then, $A$ is valid in the {\sf Int2GC+FS}-frame $(X,\leq,R)$ if and only
$A$ is valid in the {\sf Int2GC}-frame $(X,\leq,R \circ {\geq}, {\leq} \circ R)$.
\end{lemma}

\begin{theorem}[\bf Soundness II for  Int2GC+FS] \label{Thm:SoundnessII}
Every formula provable in {\sf Int2GC+FS} is Kripke-valid.
\end{theorem}

\begin{proof} Suppose that $\Up A \to B$ is valid in a 
{\sf Int2GC+FS}-frame $(X,\leq,R)$. Then, by Lemma~\ref{Lem:Validity},
$\Up A \to B$ is valid in the {\sf Int2GC}-frame $(X,\leq,R \circ {\geq}, {\leq} \circ R)$.
This implies that $A \to \DOWN B$ is valid in the {\sf Int2GC}-frame 
$(X,\leq,R \circ {\geq}, {\leq} \circ R)$, because {\sf Int2GC}-preserves
validity of the Galois connection rules. By  Lemma~\ref{Lem:Validity}, 
$A \to \DOWN B$ is valid in the {\sf Int2GC+FS}-frame $(X,\leq,R)$. 
Thus, (GC\,${\Up}{\DOWN}$) preserves validity.
Rules (GC\,${\DOWN}{\Up}$), (GC\,${\UP}{\Down}$), and 
(GC\,${\Down}{\UP}$) may be considered similarly.

We show  that axiom (D1) is a valid formula. Validity
of (D2) can be proved analogously. By Corollary~\ref{Cor:Axiomatisation},
this gives that the axioms of {\sf Int2GC+FS} are valid.

Suppose $x \models \UP A \wedge \DOWN B$. 
Then, $x \models \UP A$ and $x \models \DOWN B$. 
So, there exists $y$ such that $x(R \circ {\geq})y$ 
and $y \models A$. Thus, there is $w$ such that 
$xRv$ and $w \geq y$. Because of persistency, we have 
$w \models A$. Now $x \leq x$ and $xRw$ imply $x({\leq} \circ R)w$. 
The fact $x \models \DOWN B$ means that
for all $z$, $x ({\leq} \circ R) z$ implies 
$z \models B$. Therefore, $w \models B$ and
thus $w \models A \wedge B$. Because
$xRw$ and $w \geq w$, we have $x (R \circ {\geq}) w$
implying $x \models \UP(A \wedge B)$. So, (D1) is
a valid formula.
\end{proof}

\begin{example} \label{Exa:Preference}
We present an application showing how preference relations may be used for obtaining particular 
Kripke-frames of {\sf Int2GC+FS}.Several definitions of preference structures can be found in the 
literature; see \cite{Preference}. 
There are two fundamental relations, namely ``better'' (\emph{strict preference}) and ``similar'' 
(\emph{indifference}). Here we denote ``$b$ is better than $a$'' by $a \prec b$ and
$a \sim b$ denotes that $a$ and $b$ are similar. Usually, it is assumed that $\prec$ and $\sim$ 
have at least the following  properties:
\begin{enumerate}[(i)]
\item $a \prec b$ implies $b \nprec a$   \hfill (\emph{asymmetry of $\prec$}) 
\item $a \sim a$                    \hfill (\emph{reflexivity of $\sim$}) 
\item $a \sim b$ implies $b \sim a$ \hfill (\emph{symmetry of $\sim$}) 
\item $a \prec b$ implies $a \nsim b$   \hfill (\emph{incompatibility of $\prec$ and $\sim$}) 
\end{enumerate}

Suppose now that $\prec$ is a transitive strict preference relation on some universe of discourse $U$.
Transitivity is a quite natural property of strict preference, because if $a$ is better than $b$
and $b$ is better than $c$, also $a$ should be better that $c$. 

Let us denote by $\preceq$ the relation ${\prec} \cup {\Delta_U}$, where $\Delta_U$ is
the \emph{identity relation} of $U$, that is, $\Delta_U = \{ (x,x) \mid x \in U\}$. 
The relation $\preceq$ is obviously a preorder. Note that since $\prec$ is assumed
to be asymmetric, then $a \preceq b$ and $b \preceq a$ imply $a = b$. This means
that $\preceq$ is a partial order on $U$. Assume also that $\preceq$ and
$\sim$ are connected by conditions  \eqref{Eq:frame4} and  \eqref{Eq:frame5}, that is, 
\begin{center}
 $({\sim} \circ {\preceq}) \subseteq ({\preceq} \circ {\sim})$ \quad and \quad 
 $({\succeq} \circ {\sim}) \subseteq ({\sim} \circ {\succeq})$, 
\end{center}
where $\succeq$ is the inverse relation of $\preceq$. These assumptions
hold for instance in such object sets which can organized in ``levels''
as in Figure~\ref{Fig:Fig1} -- elements in the same level are all similar with
respect to their properties, and the elements in an upper level are better than the lower ones.
\begin{figure}[ht]
\centering
\includegraphics[width=115mm]{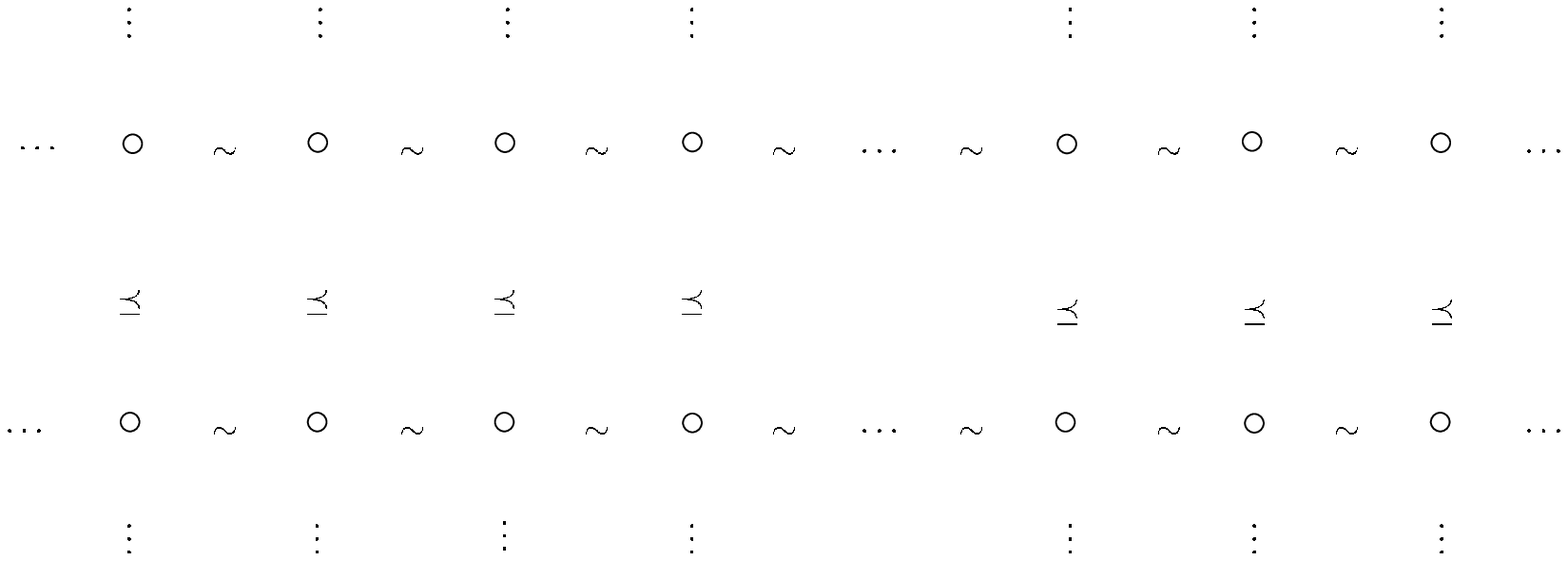} 
\caption{\label{Fig:Fig1}}
\end{figure}

Hence, the triple $(U,\preceq,\sim)$ can be viewed as an {\sf Int2GC+FS}-frame.
Because the relation $\sim$ is symmetric, $\DOWN A$ and $\Down A$ have equal interpretations, and
the same holds for $\UP A$ and $\Up A$. This
means  that ${\DOWN} A \leftrightarrow {\Down} A$ and ${\UP} A \leftrightarrow {\Up} A$
are valid formulas in any Kripke-model based on the frame $(U,\preceq,\sim)$. This
implies, for instance, that $A \to \DOWN \UP A$ and $\UP \DOWN A \to A$
are valid in all such Kripke-models for all {\sf Int2GC+FS}-formulas $A$.

Additionally, because $\sim$ and  $\preceq$ are reflexive, we have that $A \to \UP A$ and
$\DOWN A \to A$ are valid in all Kripke-models based on $(U,\preceq,\sim)$. 
Let the formula $A$ represent some property, that is,  
$x \models A$ means that the object $x \in U$ has this property. 
The formulas $\UP A$ and $\DOWN A$ have the following interpretations:
\begin{enumerate}[(i)]
\item $x \models \UP A$ if there exist $y,z \in U$ such that $x \sim y$, $y \succeq z$ and $z \models A$, 
that is, $x$ is similar to an object that is better than or equal to an object having the property $A$.

\item $x \models \DOWN A \iff $ for all $y,z$, $x \preceq y$ and $y \sim z$ imply $z \models  A$,
that is, all objects similar to the objects being better or equal to $x$ have the property $A$.
\end{enumerate}
Thus, the semantics based on preference $\succeq$ and similarity $\sim$ validates many formulas 
that are not  generally {\sf Int2GC+FS}-provable.  To get a full correspondence, one should admit
$R$ to be any \emph{information relation} satisfying \eqref{Eq:frame4} and  \eqref{Eq:frame5}, that is, 
\begin{center}
 $(R \circ {\preceq}) \subseteq ({\preceq} \circ R)$ \quad and \quad 
 $({\succeq} \circ R) \subseteq (R \circ {\succeq})$, 
\end{center}
not just symmetric ones. Various information relations are studied in \cite{DemOrl02}, for instance.
\end{example}

\subsection{Canonical Frames and Completeness of {\sf Int2GC+FS}}
\label{sSec:Canonical}

To prove the completeness theorem with respect to Kripke-models, we will apply canonical frames
and the algebraic completeness for {\sf Int2GC+FS} presented in Theorem~\ref{Thm:CompletenessI}(b).

The \emph{canonical {\sf Int2GC+FS}-frame} of an H2GC+FS-algebra $(\mathbb{H},{^\RIGHT},{^\Left}, {^\Right},{^\LEFT})$
is a structure $(X_H,\subseteq, R^H)$, where $X_H$ is the set of lattice-filters of $H$ and the
relation $R^H$ is defined by
\[ (x,y) \in R^H \iff  [x]^{\LEFT^{-1}}  \subseteq y \subseteq [x]^{\RIGHT^{-1}}. \]
The relation $R^H$ can be described also as
\[ (x,y) \in R^H \iff [y]^{\Left^{-1}} \subseteq x \subseteq [y]^{\Right^{-1}}. \]
This means that $R^H = R_1^H \cap R_2^H$, where  $R_1^H$ and $R_2^H$ are the
relations of the canonical {\sf Int2GC}-frame of an H2GC-algebra.

Next, we will show that the canonical {\sf Int2GC+FS}-frame is a Kripke-frame of {\sf Int2GC+FS}. 
Before that, we present some results and observations that are needed for our proofs. 
We denote by $[S \rangle$ the lattice-filter generated by $S \subseteq H$. It is well known that 
$[S \rangle$ is the set of all elements $a \in H$ such that $a_1 \wedge \cdots a_n \leq a$ 
for some elements $a_1,\ldots,a_n \in S$. We also denote for any $x \in X_H$:
\begin{align*}
[x]^{\LEFT}   &= \{ a^\LEFT   \mid a \in x\};             
&[x]^{\RIGHT} &= \{ a^\RIGHT \mid a \in x\};\\
[x]^{\Left}   &= \{ a^\Left  \mid a \in x\}; 
&[x]^{\Right} &= \{ a^\Right  \mid a \in x\}. 
\end{align*}

\begin{lemma}\label{Lem:FilterPrimefilter}
Let  $(\mathbb{H},{^\RIGHT},{^\Left}, {^\Right},{^\LEFT})$ be an {H2GC+FS}-algebra. 
If $k$ is a filter and  $y$ is a prime filter such that  $k \cap [-y]^\LEFT = \emptyset$, 
then there exists a prime filter $u$ such that $k \subseteq u$ and $u \cap [-y]^\LEFT = \emptyset$.
\end{lemma}

\begin{proof}
Let us denote $\Gamma = \{ t \mid \text{$t$ is a filter, $k \subseteq t$, and $t \cap [-y]^\LEFT = \emptyset$}\}$. 
Clearly $\Gamma \neq \emptyset$ and, by Zorn's Lemma, $\Gamma$ has a maximal element $u$. Then, $u$ is a filter,
$k \subseteq u$, and $u \cap [-y]^\LEFT = \emptyset$.

Assume that $u$ is not a prime filter. Then there exists two elements
$a,b \in H$ such that $a \vee b \in u$, but $a,b \notin u$. By maximality of $u$,
this implies that $[u \cup \{a\} \rangle$ and $[u \cup \{b\} \rangle$ are not in $\Gamma$. 
Therefore, we must have that $[u \cup \{a\} \rangle \cap  [-y]^\LEFT \neq \emptyset$ and 
$[u \cup \{b\} \rangle \cap  [-y]^\LEFT \neq \emptyset$.
So, there exists $c,d \in   [-y]^\LEFT$ such that $c \in [u \cup \{a\} \rangle$ and
$d \in [u \cup \{b\} \rangle$. Because $u$ is a filter, this implies that there exist $e,f \in u$
such that $e \wedge a \leq c$ and $f \wedge b \leq d$. Now $e \wedge f \in u$ and 
$a \vee b \in u$ imply $(e \wedge f) \wedge (a \vee b) \in u$. Since
\[ (e \wedge f) \wedge (a \vee b) =  (e \wedge f \wedge a)  \vee  (e \wedge f \wedge b)  \leq c \vee d,\]
we obtain $c \vee d \in u$. Now the exist $c_1,d_1 \in -y$ such that $c = {c_1}^\LEFT$ and $d = {d_1}^\LEFT$.
Because $c \vee d \in u$ and $c \vee d = {c_1}^\LEFT \vee {d_1}^\LEFT \leq (c_1 \vee d_1)^\LEFT$, we have
$(c_1 \vee d_1)^\LEFT \in u$. On the other hand, $c_1,d_1 \notin y$ implies $c_1 \vee d_1 \notin y$,
because $y$ is a prime filter. Thus, $c_1 \vee d_1 \in -y$ implies $(c_1 \vee d_1)^\LEFT \in [-y]^\LEFT$.
But  $u \cap [-y]^\LEFT = \emptyset$, a contradiction. Thus, $u$ is a prime filter.
\end{proof}

Let $S$ be a non-empty subset of a lattice $L$ such that $a \vee b \in S$ implies 
$a \in S$ or $b \in S$ for all $a,b \in L$. It is easily seen that such sets $S$
can be characterised as the sets whose set-theoretical complement $-S$ is a 
$\vee$-subsemilattice of $L$. In \cite{DzJaKo10}, we proved the following lemma.

\begin{lemma} \label{Lem:Co-filter} Let $L$ be a distributive lattice.
If $x$ is a filter and $u$ is a superset of $x$ such that its
set-theoretical complement $-u$ is a $\vee$-subsemilattice of $L$,
then there exists a prime filter $z$ such that $x \subseteq z \subseteq u$.
\end{lemma}

Our next proposition shows that the canonical frames are {\sf Int2GC+FS}-frames.

\begin{proposition}\label{Prop:CanonicalFrameIs}
If $(\mathbb{H},{^\RIGHT},{^\Left}, {^\Right},{^\LEFT})$ is an {H2GC+FS}-algebra,
then $(X^H,\subseteq,R^H)$ is an {\sf Int2GC+FS}-frame.
\end{proposition}

\begin{proof} \eqref{Eq:frame4} Assume that $x \, (R^H \circ {\subseteq}) \, y$. This implies
that there exists $z \in X_H$ such that $x \, R^H \, z$ and $z \subseteq y$.
Then, $[x]^{\LEFT^{-1}} \subseteq z \subseteq [x]^{\RIGHT^{-1}}$.
Let $k = [x \cup [y]^\RIGHT)$ be the filter generated by $x \cup [y]^\RIGHT$.
We show first that $k \cap [-y]^\LEFT = \emptyset$. Namely, if
$k \cap [-y]^\LEFT \neq \emptyset$, then there exists an element $a$
such that $a \in k$ and $a \in [-y]^\LEFT$. Since  $a \in k = [x \cup [y]^\RIGHT)$,
there are $b \in x$ (recall that $x$ is a filter) and 
$c_1,\dots,c_n \in y$ such that $b \wedge {c_1}^\RIGHT \wedge \cdots \wedge {c_n}^\RIGHT \leq a$.
Let us denote $c = c_1 \wedge \cdots \wedge c_n \in y$. Because
the map $^\RIGHT$ is order-preserving, we have 
$c^\RIGHT \leq {c_1}^\RIGHT \wedge \cdots \wedge {c_n}^\RIGHT$. This means
that $b \wedge c^\RIGHT \leq a$ and $b \leq c^\RIGHT \to a$. Now
$a = d^\LEFT$ for some $d \in -y$, and so $b \leq c^\RIGHT \to d^\LEFT \leq (c \to d)^\LEFT$
by (fs3). Since $b \in x$ and $x$ is a filter, we have $(c \to d)^\LEFT \in x$. From this
we get $c \to d \in [x]^{\LEFT^{-1}} \subseteq z \subseteq y$.
Because $c \in y$ and $y$ is a filter, also $c \wedge (c \to d) \in y$.
Since $c \wedge (c \to d) \leq d$, we have $d \in y$,
a contradiction. Hence, $k \cap [-y]^\LEFT = \emptyset$.

By Lemma~\ref{Lem:FilterPrimefilter}, there exists a prime filter $u$
such that $k \subseteq u$ and $u \cap [-y]^\LEFT = \emptyset$. So,
$x \subseteq k \subseteq u$, $[y]^\RIGHT \subseteq k \subseteq u$, 
and $y \subseteq [u]^{\RIGHT^{-1}}$.
In addition, $[u]^{\LEFT^{-1}} \subseteq y$, because if $a \in [u]^{\LEFT^{-1}}$,
then $a^\LEFT \in u$ which gives $a^\LEFT \notin [-y]^\LEFT$, because 
$u \cap [-y]^\LEFT = \emptyset$. Thus, $a \notin -y$, that is, $a \in y$.
Now $x \subseteq u$ and $u \, R^H \, y$ give $x \, ({\subseteq} \circ R^H)\, y$.

\smallskip\noindent%
\eqref{Eq:frame5} Assume $x \, ({\supseteq} \circ R^H) \, y$. Then for some  $w \in X_H$,
$x \supseteq w$ and $[w]^{\LEFT^{-1}} \subseteq y \subseteq [w]^{\RIGHT^{-1}}$. Hence,
$y \subseteq [x]^{\RIGHT^{-1}}$, because $^{\RIGHT^{-1}}$ is order-preserving.
To show that $x \, (R^H \circ {\supseteq}) \, y$, we need to find a prime filter
$z \in X_H$ such that $[x]^{\LEFT^{-1}} \subseteq z \subseteq [x]^{\RIGHT^{-1}}$
and $z \supseteq y$. Consider the filter $k = \big [y \cup [x]^{\LEFT^{-1}} \big \rangle$.
We show first that $k \subseteq [x]^{\RIGHT^{-1}}$. Assume $a \in k$.
Then, there exists $c \in y$ and $d \in [x]^{\LEFT^{-1}}$ such that 
$c \wedge d \leq a$ (note that $y$ is a filter and $[x]^{\LEFT^{-1}}$
is closed under meets). Hence, $c \leq d \to a$ and
$c^\RIGHT \leq (d \to a)^\RIGHT \leq d^\LEFT \to a^\RIGHT$ by (fs1). Since 
$c \in y \subseteq [x]^{\RIGHT^{-1}}$, we have $c^\RIGHT \in x$ and
$d^\LEFT \to a^\RIGHT \in x$. Because $d^\LEFT \in x$, we obtain 
$a^\RIGHT \in x$, that is, $a \in [x]^{\RIGHT^{-1}}$ as required.

Because $k \subseteq [x]^{\RIGHT^{-1}}$, $k$ is a filter, and $[x]^{\RIGHT^{-1}}$
is a set such that is set-theoretical complement is a $\vee$-subsemilattice of $L$, 
by Lemma~\ref{Lem:Co-filter} there exists $z \in X_H$ such that $k \subseteq z \subseteq [x]^{\RIGHT^{-1}}$.
Combining the above observations, we have $z \supseteq k \supseteq y$ and
$[x]^{\LEFT^{-1}} \subseteq k \subseteq z \subseteq [x]^{\RIGHT^{-1}}$, that is,
$x \, R^H \, z$ and $z \supseteq y$.
\end{proof}

Let $(\mathbb{H},{^\RIGHT},{^\Left}, {^\Right},{^\LEFT})$ be an H2GC+FS-algebra. 
Let $v \colon \mathrm{Var} \to H$ be a valuation. We may now define a valuation 
$v^* \colon \mathrm{Var} \to X_H$ for the canonical frame $\mathcal{F}^H = (X_H, \subseteq, R^H)$ 
by setting $x \in v^*(p)$ if and only if $v(p) \in x$ for all $p \in \mathrm{Var}$. Hence,
in the canonical model $(\mathcal{F}^H,v^*)$, we have $x \models p$ if and only if 
$v(p) \in x$ for all $x \in X_H$. 
We show that an analogous condition holds for all formulas $A$.

\begin{lemma}[\bf Key Lemma]
Let $(\mathbb{H},{^\RIGHT},{^\Left}, {^\Right},{^\LEFT})$ be an H2GC+FS-algebra
and $v \colon \mathrm{Var} \to H$ a valuation. In the canonical model $(\mathcal{F}^H,v^*)$, 
we have $x \models A$ if and only if $v(A) \in x$ for all $x \in X_H$ and $A \in \Phi$.
\end{lemma}

\begin{proof} 
We prove the result by formula induction. For the
connectives $\vee$, $\wedge$, $\to$, and $\neg$ the result
is well known from the theory of intuitionistic logic.
In addition, we only show the proofs for formulas $\Up A$ and $\DOWN A$,
since for $\UP A$ and $\Down A$ the proofs are analogous.

\smallskip\noindent%
(${\UP} A$) Suppose that $x \models \UP A$. This means that there exists
a prime filter $y$ such that $x \, (R^H \circ {\supseteq}) \, y$ and $y \models A$.
By the induction hypothesis, we have that $y \in v(A)$. In addition, there exists
a prime filter $u$ such that $x \, R^H \, u$ and $y \subseteq u$, that is,
$[x]^{\LEFT^{-1}} \subseteq u \subseteq [x]^{\RIGHT^{-1}}$. We obtain directly
that $v(A) \in y \subseteq u \subseteq [x]^{\RIGHT^{-1}}$, which means
that $v(A)^\RIGHT = v (\UP A) \in x$.

Conversely, suppose $v(\UP A) = v(A)^\RIGHT \in x$. Let us consider the filter
$k = \big [[x]^{\LEFT^{-1}} \cup \{v(A)\} \big \rangle$. First we show that
$k \subseteq [x]^{\RIGHT^{-1}}$. Assume that $a \in k$. Then there
exists $b \in [x]^{\LEFT^{-1}}$ such that $b \wedge v(A) \leq a$
(note that $[x]^{\LEFT^{-1}}$ is closed under finite meets). We have that
$v(A) \leq b \to a$ and $v(A)^\RIGHT \leq (b \to a)^\RIGHT \leq b^\LEFT \to a^\RIGHT$
by (fs1). This implies $b^\LEFT \to a^\RIGHT \in x$. Because $b^\LEFT \in x$,
we obtain $a^\RIGHT \in x$ and $a \in [x]^{\RIGHT^{-1}}$. Hence, 
$k \subseteq [x]^{\RIGHT^{-1}}$. Because $k$ is a filter and $[x]^{\RIGHT^{-1}}$
is a set such that is set-theoretical complement is a $\vee$-subsemilattice of $H$, 
we have by Lemma~\ref{Lem:Co-filter} that there exists a prime filter $y$ such that 
$k \subseteq y \subseteq [x]^{\RIGHT^{-1}}$. By the definition of $k$,
$[x]^{\LEFT^{-1}} \subseteq k \subseteq y$ and $v(A) \in k \subseteq y$.
We have $[x]^{\LEFT^{-1}} \subseteq y \subseteq [x]^{\RIGHT^{-1}}$,
that is, $x \, R^H \, y$. By the induction hypothesis, $y \models A$.
Since $y \supseteq y$ holds trivially, we have
$x \, (R^H \circ {\supseteq}) \, y$ implying $x \models \UP A$.

\smallskip\noindent%
(${\Down} A$) Suppose that $v(\Down A) = v(A)^\Left \in x$. 
Let $y \in X_H$. If $y \, (R^H \circ {\supseteq}) \, x$,
then there exists $z$ such that $y \, R^H \, z$ and $z \supseteq x$. 
Now $y \, R^H \, z$ is equivalent to 
$[z]^{\Left^{-1}} \subseteq y \subseteq [z]^{\Right^{-1}}$. 
Therefore, $v(A)^\Left \in x \subseteq z$ gives 
$v(A) \in [z]^{\Left^{-1}} \subseteq y$. By the induction
hypothesis, $y \models A$ and hence $x \models \Down A$.

For the other direction, assume $v(\Down A) = v(A)^\Left \notin x$,
that is, $v(A) \notin [x]^{\Left^{-1}}$. It is easy to observe
that $[x]^{\Left^{-1}}$ is a filter. Then, by the Prime Filter
Theorem of distributive lattices (see 
\cite{DzJaKo10}*{Lemma~5.4}, for instance),
there exists a prime filter $u$ such that
$v(A) \notin u$ and $[x]^{\Left^{-1}} \subseteq u$.

Let us consider the filter $k = [x \cup [u]^\Right \rangle$.
We first show that $k \cap [-u]^\Left = \emptyset$.
If $k \cap [-u]^\Left \neq \emptyset$, then there
exists $a \in k \cap [-u]^\Left$. Because
$a \in k = [x \cup [u]^\Right \rangle$, there are
$b \in x$ and $c_1,\ldots,c_n \in u$ such
that $b \wedge {c_1}^\Right \wedge \cdots \wedge {c_n}^\Right \leq a$
(recall that $x$ is a filter). Let us denote
$c = c_1 \wedge \cdots \wedge c_n \in u$.
Hence, $c^\Right \leq  {c_1}^\Right \wedge \cdots \wedge {c_n}^\Right$
and $b \wedge c^\Right \leq a$. But now $a = d^\Left$
for some $d \notin u$. So, $b \wedge c^\Right \leq d^\Left$.
This gives that $b \leq c^\Right \to d^\Left \leq (c \to d)^\Left$
by (fs4). Because $b \in x$, we have $(c \to d)^\Left \in x$.
This means that $c \to d \in [x]^{\Left^{-1}} \subseteq u$.
Now $c \in u$ implies $d \in u$, a contraction.
Therefore, $k \cap [-u]^\Left = \emptyset$.

By Lemma~\ref{Lem:FilterPrimefilter}, there exists a prime filter $y$ such that 
$k \subseteq y$ and $y \cap [-u]^\Left = \emptyset$. Since $k \subseteq y$,
we have $x \subseteq y$ and $[u]^\Right \subseteq y$ meaning 
$u \subseteq [y]^{\Right^{-1}}$. The fact that $y \cap [-u]^\Left = \emptyset$ 
implies $[y]^{\Left^{-1}} \subseteq u$, because if $a \in [y]^{\Left^{-1}}$,
then $a^\Left \in y$. This gives $a^\Left \notin [-u]^\Left$, 
$a \notin -u$, and $a \in u$. By combining our observations,
we have $[y]^{\Left^{-1}} \subseteq u \subseteq [y]^{\Right^{-1}}$,
that is, $u \, R^H \, y$ and $y \supseteq x$. 
Thus, $u \, (R^H \circ {\supseteq}) \, x$. Because $v(A) \notin u$,
we have $u \not \models A$ by the induction hypothesis.
Hence, $x \not \models \Down A$.
\end{proof}

As in case of Theorem~\ref{Thm:Completeness2}, we may prove the following
completeness result by applying the Key Lemma.

\begin{theorem}[\bf Completeness~II for Int2GC+FS] \label{Thm:Completeness3}
A formula $A$ is provable in  {\sf Int2GC+FS} if and only if $A$ is Kripke-valid.
\end{theorem}

\section{Connections to intuitionistic tense logic} \label{Sec:IKt}

Intuitionistic tense logic {\sf IK$_t$} was introduced by Ewald \cite{Ewald86} 
by extending the language of intuitionistic propositional logic with the usual
temporal expressions $FA$ ($A$ is true at some future time),
$PA$ ($A$ was true at some past time), 
$GA$ ($A$ will be true at all future times), and
$HA$ ($A$ has always been true in the past).
The Hilbert-style axiomatisation of {\sf IK$_t$} can be found in \cite{Ewald86}*{p.~171}:

\begin{tabbing}
\TABS
\>\ (1) \ All axioms of intuitionistic logic \\
\>\ (2) \ $G(A \to B) \to (GA \to GB)$ 				\>\ (2$'$) \ $H(A \to B) \to (HA \to HB)$ \\
\>\ (3) \ $G(A \wedge B) \leftrightarrow GA \wedge GB$ 		\>\ (3$'$) \ $H(A \wedge B) \leftrightarrow HA \wedge HB$ \\  
\>\ (4) \ $F(A \vee B) \leftrightarrow FA \vee FB$ 		\>\ (4$'$) \ $P(A \vee B) \leftrightarrow PA \vee PB$ \\
\>\ (5) \ $G(A \to B) \to (FA \to FB)$ 				\>\ (5$'$) \ $H(A \to B) \to (PA \to PB)$ \\ 
\>\ (6) \ $GA \wedge FB \to F(A \wedge B)$			\>\ (6$'$) \ $HA \wedge PB \to P(A \wedge B)$ \\
\>\ (7) \ $G \neg A  \to \neg FA$				\>\ (7$'$) \ $H \neg A \to \neg PA$ \\
\>\ (8) \ $FHA \to A$						\>\ (8$'$) \ $PGA \to A$ \\
\>\ (9) \ $A \to HFA$ 						\>\ (9$'$) \ $A \to GPA$ \\
\>$\!$(10) \ $(FA \to GB) \to G(A \to B)$			\>$\!$(10$'$) \ $(PA \to HB) \to H(A \to B)$ \\
\>$\!$(11) \ $F(A \to B) \to (GA \to FB)$			\>$\!$(11$'$) \ $P(A \to B) \to (HA \to PB)$ \\
\end{tabbing}
The rules of inference are modus ponens (MP), and 
\begin{tabbing}
\TABS
\>(RH) \ $\displaystyle \frac{A}{H A}$
\>
$\!$(RG) \ $\displaystyle \frac{A}{G A}$
\end{tabbing}\medskip

Our next proposition shows that if we identify $\UP$, $\DOWN$, $\Up$, $\Down$ with
$F$, $G$, $P$, $H$, respectively, then {\sf Int2GC+FS} and {\sf IK$_t$} will become  syntactically equivalent.
Recall that 
\[ \mbox{\sf Int2GC+FS} = {\sf Int2GC} + \{ {\rm (FS1)} \ {\rm or} \ {\rm (FS4)} \ {\rm or} \ {\rm (D1)} \}
+ \{ {\rm (FS2)}  \ {\rm or} \ {\rm (FS3)} \ {\rm or} \ {\rm (D2)} \}, \]
and {\sf Int2GC} is obtained by extending intuitionistic logic with rules (GC\,${\Down}{\UP}$), 
(GC\,${\UP}{\Down}$), (GC\,${\DOWN}{\Up}$), and (GC\,${\Up}{\DOWN}$).

\begin{theorem}\label{Thm:Equivalence}
$\text{\sf IK}_t = \text{\sf Int2GC+FS}$.
\end{theorem}

\begin{proof} 
First we show that all axioms {\sf IK$_t$} are provable in {\sf Int2GC+FS}, and all rules
of {\sf IK$_t$} are admissible in {\sf Int2GC+FS}. In this first part, let $\vdash$ denote
that a formula $A$ is provable in  {\sf Int2GC+FS}. As noted in Section~\ref{Sec:Axiomatizations},
axioms (2), (2$'$), (3), (3$'$) (4), (4$'$), (8), (8$'$), (9), (9$'$) are provable even in {\sf Int2GC}.
Additionally, rules (MP), (RH), and (RG) are admissible in {\sf Int2GC}.
Axioms (10), (10$'$), (11), (11$'$) are the Fischer Servi axioms (FS3), (FS4), (FS1), (FS2), so
they are provable in {\sf Int2GC+FS}.
                                                                                           
Axiom (FS1) is equivalent to $\UP(A \to B) \wedge \DOWN A \to \UP B$. If we set $B := A \wedge B$ in this formula,  
we have that $\vdash (\UP(A \to A \wedge B) \wedge \DOWN A ) \to \UP (A \wedge B)$. 
Because $A \to A \wedge B$ is equivalent to $A \to B$, and $\vdash B \to (A \to B)$ gives
$\vdash \UP B \to \UP(A \to B)$ by the monotonicity of $\UP$, we obtain $\vdash \UP B \wedge \DOWN A \to \UP(A \wedge B)$
and thus (6) is provable in {\sf Int2GC+FS}. Provability of (6$'$) can be shown similarly.

Because $\vdash \UP \Down (A \to B) \to (A \to B)$, we have $\vdash \UP \Down (A \to B) \wedge A \to B$
and $\vdash \Up(\UP \Down (A \to B) \wedge A)  \to \Up B$.
Let us set $A := \UP \Down(A \to B)$ and $B := A$ in axiom (6$'$) (which we just showed to be provable
in {\sf Int2GC+FS}). We obtain $\vdash \Down \UP \Down (A \to B) \wedge \Up A \to \Up(\UP \Down (A \to B) \wedge A)$. 
Thus, $\vdash \Down \UP \Down (A \to B) \wedge \Up A \to \Up B$.
Because $\vdash \Down(A \to B) \to \Down \UP \Down(A \to B)$, we have 
$\vdash \Down(A \to B) \wedge \Up A \to \Up B$. This is equivalent to
$\vdash \Down(A \to B) \to (\Up A \to \Up B )$. Hence, (5$'$) is provable in 
{\sf Int2GC+FS}, and provability of (5) can be showed in an analogous manner.
 
If we set $B := \bot$ in (5), we get $\vdash \DOWN(A \to \bot) \to (\UP A \to \UP \bot )$. 
Because $\UP \bot$ is equivalent to $\bot$, we have $\vdash \DOWN \neg A \to \neg \UP A$.
This means that (7) and (7$'$) are provable.

Because axioms (10), (10$'$), (11), (11$'$) are the Fischer Servi axioms, for the other direction
is enough to show admissibility of rules (GC\,${\Down}{\UP}$), (GC\,${\UP}{\Down}$), (GC\,${\DOWN}{\Up}$), 
(GC\,${\Up}{\DOWN}$) in {\sf IK$_t$}. First, we show admissibility of the rules of monotonicity, 
that is, if $A \to B$ is provable, then $H A \to H B$, $PA \to P B$, $G A \to G B$, and $FA \to F B$ are provable.

Here $\vdash A$ denotes that the formula $A$ is provable in {\sf IK$_t$}.
Assume  $\vdash A \to B$. By (RG), $\vdash G(A \to B)$.
Now $\vdash G A \to G B$ follows by (2), and from $\vdash G(A \to B)$,
we obtain also $\vdash FA \to FB$ by (5). Similarly, $\vdash A \to B$ implies 
$\vdash HA \to H B$ and $\vdash P A \to P B$ by applying (RH), (2$'$), and (5$'$).

Next we prove admissibility of (GC\,${\Down}{\UP}$). Assume that  $\vdash A \to H B$. 
Then,  $F A \to F H B$ by monotonicity of $F$. Because $\vdash F H B \to B$ by (8),
we obtain $\vdash F A \to B$.  Similarly, by (8$'$) and monotonicity of $P$, $A \to GB$ implies 
$PA \to B$, that is,  (GC\,${\DOWN}{\Up}$) is admissible in {\sf IK$_t$}.
Monotonicity of $H$ and axiom (9) yield $FA \to B$ implies $A \to HB$,
and monotonicity of $G$ and (9$'$) give that $PA \to B$ implies $A \to BG$. Thus, rules
(GC\,${\UP}{\Down}$) and (GC\,${\Up}{\DOWN}$) are admissible. 
\end{proof}

\begin{remark}  It is proved in \cite{JaKoKo08} that {\sf ILGC} is equivalent, with respect to provability, 
to the minimal (classical) tense logic {\sf K$_t$}, that is, {\sf ILGC} can be viewed as a simple formulation of {\sf K$_t$}. 
The same analogy applies here, because {\sf Int2GC+FS} can be seen as an alternative formulation of {\sf IK$_t$}.

It should be noted that with respect to Kripke-semantics, {\sf IK$_t$} and {\sf Int2GC+FS} are quite
different. A Kripke-frame of {\sf IK$_t$} consists of  a partially-ordered set
$(\Gamma, \leq)$ (the ``states-of-knowledge''), family of sets $T_\gamma$, where $\gamma \in \Gamma$ (times known at 
state-of-knowledge $\gamma$), such that $\gamma \leq \varphi$ implies $T_\gamma \subseteq T_\varphi$, meaning that 
advancing in knowledge retains what is known about times and their temporal ordering, and a collection of
binary relations $\mu_\gamma$ on $T_\gamma$ (the temporal ordering of $T_\gamma$ as it is 
understood at state-of-knowledge $\gamma$) \cite{Ewald86}, whereas
{\sf Int2GC+FS} is conceived as an \emph{information logic} such that its frames $(X,\leq,R)$ are
such that $X$ forms the \emph{universe of discourse}, and $\leq$ and $R$ are  relations
reflecting relationships between the objects in $X$, such as preference and indifference
of objects (see Example~\ref{Exa:Preference}).
\end{remark}

It is also obvious and well-known that the axiomatisation of Ewald is not minimal, because
several axioms can be deduced from the other axioms. We present a reduced axiomatisation,
in which the number of axioms is the half of the size of the axiomatisation in \cite{Ewald86}.

\begin{proposition}\label{Prop:MinAxs}
{\sf IK$_t$} can be axiomatised by adding {\rm (2), (2$'$), (5), (5$'$), (8), (8$'$), (9), (9$'$), (11), (11$'$)}  
to the axioms of intuitionistic logic together with rules {\rm (MP)}, {\rm (RH)}, and {\rm (RG)}. 
\end{proposition}

\begin{proof}
As shown in the proof of Theorem~\ref{Thm:Equivalence}, if $\UP$, $\DOWN$, $\Up$, $\Down$ are identified 
with $F$, $G$, $P$, $H$, then axioms (2), (2$'$), (5), (5$'$), (8), (8$'$), (9), (9$'$) with rules 
(MP), (RH) and (RG) are enough to show that rules (GC\,${\Down}{\UP}$), (GC\,${\UP}{\Down}$), 
(GC\,${\DOWN}{\Up}$), (GC\,${\Up}{\DOWN}$) are admissible. 
Axioms (11), (11$'$) coincide with (FS1) and (FS2), so the proof is complete,
because $\text{\sf IK}_t = \text{\sf Int2GC+FS}$.
\end{proof}

In the next proposition, we present another axiomatisation of {\sf Int2GC} using axioms of intuitionistic tense logic. 

\begin{proposition}\label{Prop:Int2GCAX}
{\sf Int2GC} can be axiomatised by adding {\rm (2), (2$'$), (8), (8$'$), (9), (9$'$)}  
to the axioms of intuitionistic logic together with rules {\rm (MP)}, {\rm (RH)}, {\rm (RG)},
and rules:
\begin{tabbing}
\TABS
\>{\rm (RMF)} \ $\displaystyle \frac{A \to B}{F A \to FB}$
\>{\rm (RMP)} \ $\displaystyle \frac{A \to B}{P A \to PB}$\\
\end{tabbing}
\end{proposition}

\begin{proof} Monotonicity of $G$ and $H$ follow from (RG), (RH), (2), and (2$'$).
Because all operators are thus monotone, admissibility of rules
(GC\,${\Down}{\UP}$), (GC\,${\UP}{\Down}$), (GC\,${\DOWN}{\Up}$), (GC\,${\Up}{\DOWN}$)
follow easily from (8), (8$'$), (9), (9$'$).

On the other hand, in Section~\ref{Sec:Axiomatizations} we have noted that
axioms  {\rm (2), (2$'$), (8), (8$'$), (9), (9$'$)}  are provable in
{\sc Int2GC} and rules (RH), (RG), (RMF), (RMP) are admissible.
\end{proof}

Let us observe that {\sf Int2GC} cannot be axiomatised by using only axioms 
and rules of Ewald's system. The reason for this is that monotonicity of operators $P$ and $F$ need to be added,
since rules (RMF) and (RMP) do not belong to the system by Ewald as ``initial rules'', even they are admissible
in {\sf IK$_t$}. On the other hand, monotonicity of 
$P$ and $F$ could be obtained by adding axioms (5) and (5$'$) to the system of Proposition~\ref{Prop:Int2GCAX} 
(without monotonicity of $F$ and $P$), but then we have a logic which is too strong, since (5) and (5$'$) cannot be proved in 
{\sf Int2GC} -- this is because Galois connections $({\UP},{\Down})$ and $({\Up},{\DOWN})$ are ``independent'',
that is, operations $\UP$ and $\DOWN$  are not in anyway connected with each other. 
For instance, consider an H2GC-algebra on the three element chain $0 < u < 1$ such that
$^\RIGHT$ and $^\Left$ equal the identity mapping, and $x^\Right = 0$ and $x^\LEFT = 1$ for all $x \in \{0,u,1\}$.
Then $(1 \to u)^\LEFT = 1$, but $1^\RIGHT \to u^\RIGHT = 1 \to u = u$.
This actually means that we have an ``intermediate logic'' 
$\text{\sf Int2GC} + \{(5),(5')\}$  situated between {\sf Int2GC} and {\sf Int2GC+FS}. 
However, the study of $\text{\sf Int2GC} + \{(5),(5')\}$ is confined outside of the scope of this work.


\end{document}